\documentclass[11pt]{article} \usepackage{fullpage}
\usepackage{amsmath,mathtools,amsthm,xcolor,verbatim,yhmath,url,upgreek,setspace}
\usepackage{amssymb, amsfonts} 
\usepackage{graphicx, wrapfig, float,
  caption, subcaption} 
\usepackage{mathrsfs} \usepackage{hyperref}
\usepackage{graphicx}
\usepackage{enumerate}

\newcommand{\real}{\mathbb{R}}

\newcommand{\eval}{\mathbb{E}\,}

\newcommand{\inprod}[2]{\left\langle#1,#2\right\rangle}
\newcommand{\absinprod}[2]{\left|\left\langle#1,#2\right\rangle\right|}
\newcommand{\abs}[1]{\lvert#1\rvert}
\newcommand{\norm}[1]{\left\lVert#1\right\rVert}

\newtheorem{thm}{Theorem}[section]
\newtheorem{cor}[thm]{Corollary}
\newtheorem{lemma}[thm]{Lemma}
\newtheorem{prop}[thm]{Proposition}

\theoremstyle{definition}

\theoremstyle{remark}
\newtheorem{remark}[thm]{Remark}

\allowdisplaybreaks

\numberwithin{equation}{section}

\begin{document}

\title{From intersection bodies to dual centroid bodies: a stochastic
  approach to isoperimetry}





\author{Rados{\l}aw Adamczak\thanks{Supported by the National Science
    Center, Poland via the Sonata Bis grant no. 2015/18/E/ST1/00214.}
  \and Grigoris Paouris\thanks{Supported by NSF Grant DMS 1800633 and
    Simons Foundation Fellowship \#823432.}  \and Peter
  Pivovarov\thanks{Supported by NSF Grant DMS-2105468 and Simons
    Foundation Grant \#635531.} \and Paul Simanjuntak$^\ddag$}


\maketitle 
{\let\thefootnote\relax}  
{\let\thefootnote\relax}


\abstract{We establish a family of isoperimetric inequalities for sets
  that interpolate between intersection bodies and dual $L_p$ centroid
  bodies. This provides a bridge between the Busemann intersection
  inequality and the Lutwak--Zhang inequality. The approach depends on
  new empirical versions of these inequalities. }

\section{Introduction}

The focus of this paper is on connections between fundamental
inequalities in Brunn-Minkowski theory and dual Brunn-Minkowski
theory. The former details the behavior of the volume of Minkowski
sums of convex bodies. The standard isoperimetric inequality is
emblematic of deep principles within Alexandrov's theory of mixed
volumes \cite{Schneider_book}. A central line of research is on
affine-invariant strengthenings of kindred isoperimetric principles,
especially around projections of convex sets; as a sample, see
Lutwak's survey \cite{Lut93}, Schneider's monograph
\cite{Schneider_book}, the fundamental papers
\cite{LYZ00,LYZ_Orlicz,LYZ10_Proj}, and \cite{MilYeh} for a recent
breakthrough. In dual Brunn-Minkowski theory, the emphasis is on
star-shaped sets and radial addition. Dual mixed volumes, put forth by
Lutwak \cite{Lut75}, parallel many aspects of mixed volumes.  They
provide a rich framework for studying intersections of star bodies
with subspaces; for example, see \cite{Lut79,Lutwak88} for
foundational results; the monographs by Koldobsky \cite{Kol05} and
Gardner \cite{Gar06} for the resolution of the Busemann-Petty problem
and interplay with geometric tomography; the papers
\cite{HLYZ,HLYZ18,BLYZZ19, Bor} for striking new
developments. Establishing an important family of isoperimetric
inequalities linking the two theories has remained a principle
challenge.

A common root for the inequalities we treat is the {\it Busemann
  intersection inequality} \cite{Bus53} for the volume of central
slices of a compact set $K\subseteq \mathbb{R}^n$:
\begin{equation}
  \label{eqn:BusInt}
  \int_{S^{n-1}}\lvert K\cap u^{\perp}\rvert^ndu\leq
  \frac{\omega_{n-1}^{n}}{\omega_n^{n-1}}\abs{K}^{n-1},
\end{equation}
where $du$ denotes integration with respect to the normalized Haar
probability measure on the sphere $S^{n-1}$, $\abs{\cdot}$ is volume
and $\omega_n$ is the volume of the Euclidean unit ball $B_2^n$.  The
result itself (with hindsight) is an invariant inequality for the
volume of the {\it intersection body} $I(K)$ of $K$, which is defined
by its radial function via $\rho(I(K),u) = \abs{K\cap u^{\perp}}$ (see
\S \ref{section:prelim} for notation and definitions).  Intersection
bodies were introduced by Lutwak \cite{Lutwak88} in connection with
the Busemann-Petty problem and play a crucial role in dual
Brunn-Minkowski theory \cite{Gar06,Kol05}. The proof of
\eqref{eqn:BusInt} used an essential ingredient known as the Busemann
random simplex inequality, which says that the expected volume of
certain random simplices in a convex body are minimal for ellipsoids.
Petty used the latter to establish a conjecture of Blaschke on the
volume of centroid bodies \cite{Pet61}, which is now known as the
Busemann-Petty centroid inequality.  Geometrically, given an
origin-symmetric convex body $K$ in $\mathbb{R}^n$, the centroids of
halves of $K$ cut by hyperplanes through the origin form the surface
of its centroid body.  Centroid bodies are zonoids i.e., Hausdorff
limits of Minkowski sums of segments, and thus naturally belong to
Brunn-Minkowski theory. Zonoids play an important role in functional
analysis and related fields, e.g.,
\cite{Bol69,SchWei83,BLM89,MilPaj89}.

Lutwak raised the question of connecting the Busemann intersection
inequality and the Busemann-Petty centroid inequality in \cite{Lut93}.
The latter is one of several fundamental results that lead to
strengthenings of the standard isoperimetric inequality; in
particular, it is equivalent to an inequality of Petty \cite{Petty67}
on {\it polar projection bodies}, as shown in \cite{Lut93}. Projection
bodies are also zonoids and play a central role in Brunn-Minkowski
theory \cite{Gar06}.

A functional analytic perspective has shaped the development of both
intersection bodies and polar projection bodies.  Early work in the
isometric theory of Banach spaces, going back to L\'{e}vy, introduced
stable laws in connection with embeddings in $L_p$ for
$p\in(0,2]$. Positive definite distributions, stable laws and
  associated change of density arguments play a central role
  \cite{Boch32,Schoen38,Nik70,Maur72}.  Koldobsky developed a parallel
  theory, based on a Fourier-analytic approach, for embedding in
  $L_{p}$, for $p<0$. This led to fundamental characterizations of
  intersection bodies and their higher-dimensional analogues
  \cite{Kol98,Kol00,Kol05}. With this view, intersection bodies are
  unit ``balls'' of finite-dimensional subspaces of $L_{-1}$. At the
  other end, polar projection bodies arise naturally as unit balls of
  subspaces of $L_1$ \cite{Bol69}.  In between $L_{-1}$ and $L_1$ is a
  continuum of spaces that are no longer Banach spaces.  A result of
  Koldobsky shows that the classes in between decrease as $p$ varies
  from $-1$ to $1$; in particular, every polar projection body is an
  intersection body \cite{Kol99,Kol05}.  A longstanding question of
  Kwapie\'{n} from 1970 \cite{Kwa70}, in geometric form, asks if every
  intersection body is isomorphic to a polar projection body, which
  remains unsolved; see work of Kalton and Koldobsky \cite{KalKol05}
  for progress on this question.

A rich theory of isoperimetric inequalities has flourished around
centroid bodies and polar projection bodies. Two fundamental papers in
this development are those of Lutwak--Zhang \cite{LZ97} and
Lutwak--Yang--Zhang \cite{LYZ00}. For a star-shaped body $K$ and
$1\leq p \leq \infty$, the $L_p$ centroid body $Z_p (K)$ is defined by
its support function (see \S \ref{section:prelim}) via
\begin{equation}
  \label{eqn:Lp_cent_support}
  h^p({Z_p (K)},u) = \frac{1}{|K|} \int_K
  \left|\inprod{x}{u}\right|^p \, dx.
\end{equation}
\label{eqn:LZ}
Lutwak and Zhang proved that for $1 \leq p \leq \infty$,
\begin{equation}
  \label{eqn:LZ}
  \abs{Z_p^\circ (K)} \leq \abs{Z_p^\circ (K^*)};
\end{equation}
here $K^*$ is the dilate of the unit ball centered at the origin of
the same volume as $K$. When $p=\infty$, \eqref{eqn:LZ} is the
Blaschke-Santal\'{o} inequality, which is equivalent to the affine
isoperimetric inequality \cite{Lut93}. When $p=1$, \eqref{eqn:LZ}
follows from the Busemann-Petty centroid inequality. Lutwak, Yang, and
Zhang \cite{LYZ00} later proved a stronger inequality for $Z_p(K)$
itself. These are central results within the framework of
$L_p$-Brunn-Minkowski theory, which is governed by a different
elemental notion of summation, called $L_p$-addition
\cite{Firey,Lutwak93JDG,Lutwak96}. This theory provides a basis for
wide-ranging inequalities in geometry, analysis, and probability,
e.g., \cite{LYZ02JDG,LYZ04,LYZ06IMRN, HabSch09, HabSch09JFA,HJM}.
Campi and Gronchi developed an alternate approach to isoperimetric
inequalities for $L_p$-centroid bodies in \cite{CG02,CG02polar}.  In
particular, they further developed the notion and applications of
shadow systems, as introduced by Rogers and Shephard \cite{RS58}.
These systems generalize Steiner symmetrization and have far-reaching
extensions and applications; see, e.g., \cite{CG06,CG06vol}.  There is
significant interest in $L_p$-Brunn-Minkowski theory for the
challenging setting of $p<1$ \cite{BLYZ12}; see the survey \cite{Bor},
and recent advances in \cite{KM22,Milman_CADG}, and the references
therein.

A common framework for polar projection bodies and intersection bodies
has been pursued from several perspectives.  Drawing on \cite{LZ97},
the notion of the dual $L_p$-centroid body was extended by Gardner
and Giannopoulos in \cite{GarGia99} to $p\in(-1,1)$ via
\begin{equation}
\rho^{-p}(Z^{\diamondsuit}_{p}(K),u)=\frac{1}{\abs{K}}\int_{K}
\absinprod{x}{u}^pdx.  
\end{equation}
The bodies $Z^{\diamondsuit}_p(K)$ interpolate between intersection
bodies and polar $L_p$-centroid bodies using
\begin{equation*}
  \rho(I(K), u) = \lvert K\cap u^{\perp}\rvert = \lim_{p\rightarrow
    -1^{+}} \frac{p+1}{2}\int_{K}\absinprod{x}{u}^{p}dx;
\end{equation*}
see \cite{GarGia99,GZ99,Kol05,Haberl08}.  For $p<1$,
$Z_{p}^{\diamondsuit}(K)$ need not be convex, which we emphasize here
by the use of the $\cdot^\diamondsuit$ notation. Busemann-Petty type
volume comparison problems for $Z_{p}^{\diamondsuit}(K)$, motivated by
earlier work of Grinberg and Zhang \cite{GZ99} and Lutwak
\cite{Lutwak90} were treated by Yaskin and Yaskina in \cite{YY06}.
For $p<0$, these bodies have also been termed $L_p$-intersection
bodies and characterizations of such operators as radial valuations
were established by Haberl and Ludwig \cite{HL06}; see also
\cite{Haberl09} for $p>-1$. Properties of $L_p$-intersection bodies
were further developed by Haberl in \cite{Haberl08}.  When $K$ is an
origin-symmetric convex body, a result of Berck \cite{Ber09} shows
that $Z_p^{\diamondsuit}(K)$ is actually convex for $-1< p < 1$, which
extends Busemann's seminal result for intersection bodies
\cite{Bus49}. 

We develop methods to bridge the gap between the Busemann intersection
inequality \eqref{eqn:BusInt} and the Lutwak--Zhang theorem
\eqref{eqn:LZ}.  Each of these can be proved using Steiner
symmetrization, but in very different ways.  The former applies to the
star bodies $I(K)$ and uses integral geometric identities (of
Blaschke-Petkantschin type) that are particular to slices of $K$.  The
latter relies on convexity of the polar centroid bodies
$Z_{p}^{\circ}(K)$ for $p\geq 1$.  We develop a new approach that
applies to star bodies in between these two classes, that sees
\eqref{eqn:BusInt} and \eqref{eqn:LZ} from the same viewpoint.  We
will show that \eqref{eqn:BusInt} is one of a large family of
inequalities for unit balls of finite-dimensional subspaces of $L_p$.
We merge several techniques that have been used for $p=\pm 1$. These
include symmetrization, embedding via random linear operators, and a
classical change of density technique used in Koldobsky's Fourier
analytic treatment of intersection bodies.

We follow a probabilistic approach in which $L_p$-centroid bodies are
attached to probability densities rather than sets. This view was put
forth by the second-named author \cite{Pao06} in the study of
high-dimensional measures and their concentration properties; see also
\cite{KM12,LW08}. Fundamental inequalities of Lutwak, Yang and Zhang,
in \cite{LYZ00}, were extended to probability measures by the second
and third-named authors in \cite{Pao12, PP12}.  An empirical
approach to dual $L_p$-centroid bodies, for $p\geq 1$ was developed in
further joint work with Cordero-Erausquin and Fradelizi
\cite{CEFPP15}, motivated by \cite{CG06}.  To fix the notation, we set
\begin{equation*}
  \mathcal{P}_n=\left\{f:\mathbb{R}^n\rightarrow [0,\infty) \Big\vert
    \int_{\mathbb{R}^n}f(x)dx =1, \norm{f}_{\infty}<\infty\right\},
\end{equation*}
where $\norm{f}_{\infty}$ denotes the essential supremum. For $f\in
\mathcal{P}_n$, the empirical $L_p$ centroid body
$\mathcal{Z}_{p,N}(f)$ is defined by its support function via
\begin{equation}
  \label{eqn:Zp1}
  h^p(\mathcal{Z}_{p,N}(f),u)= \frac{1}{N} \sum_{i=1}^N
  \absinprod{X_i}{u}^p,
\end{equation}
where $X_1,\ldots,X_N$ are independent random vectors with density
$f$. In \cite{CEFPP15}, a stronger stochastic version of
\eqref{eqn:LZ} was established for radial measures $\nu$ with
decreasing densities,
\begin{equation}
  \label{eqn:CEFPP_dBP}
  \mathbb{E}\nu\left(\mathcal{Z}^{\circ}_{p,N}(f)\right) \leq
  \mathbb{E}\nu(\mathcal{Z}^{\circ}_{p,N}(f^*)),
\end{equation}
where $f^*$ is the symmetric decreasing rearrangement of $f$ (see
\S \ref{section:prelim}).  By the law of large numbers,
\eqref{eqn:CEFPP_dBP} implies the Lutwak--Zhang inequalities
\eqref{eqn:LZ} when $N\rightarrow \infty$ and $f=\frac{1}{\abs{K}}\chi_K$.

The empirical inequality \eqref{eqn:CEFPP_dBP} follows from a general
theorem about random operators acting in normed spaces \cite{CEFPP15}.
The random operator viewpoint is from the asymptotic theory of normed
spaces. In seminal work, Gluskin used random operators to construct
counter-examples to a longstanding question on the maximal
Banach-Mazur distance between finite-dimensional spaces
\cite{Gluskin}. The expository article of Mankiewicz and
Tomczak-Jaegermann \cite{MTJ03} details its far-reaching extensions in
Banach space theory.  This viewpoint was also fruitful in developing
stochastic versions of a number of isoperimetric inequalities
\cite{PP12,PP17}. However, inherent in the method was a restriction to
     {\it convex} sets. The main new feature we develop here is its
     applicability to {\it star-shaped} sets. We will show how this
     change provides a bridge between the aforementioned inequalities
     in Brunn-Minkowski theory and dual Brunn-Minkowski theory.

\section{Main results}

Our first result establishes a sharp isoperimetric inequality that
extends the Lutwak--Zhang inequality \eqref{eqn:LZ} to the case $p\in
(0,1)$. For $f\in \mathcal{P}_n$ and $p\in (0,1)$, define the dual
$L_p$-centroid body $Z_p^{\diamondsuit}(f)$ via its radial function:
\begin{equation}
  \label{eqn:rhoZp}
  \rho^{-p}(Z_p^{\diamondsuit}(f),u) =
  \int_{\mathbb{R}^n}\absinprod{x}{u}^{p}f(x)
  dx.
\end{equation}
To define the empirical version $\mathcal{Z}_{p,N}^{\diamondsuit}(f)$,
we let $N>n$ and consider independent random vectors $X_1,\ldots,X_N$
according to $f$ as above, and set \begin{equation}
  \label{eqn:rhoZpn}
\rho^{-p}(\mathcal{Z}_{p,N}^{\diamondsuit}(f),u) = \frac{1}{N}
\sum_{i=1}^N \absinprod{X_i}{u}^p.
\end{equation}

\begin{thm}
  \label{thm:Zp+}
  Let $f\in \mathcal{P}_n$ and let $0<p<1$. Then
  \begin{equation}
    \label{eqn:Zp+}
    \abs{Z_{p}^{\diamondsuit} (f)} \leq \abs{Z_{p}^{\diamondsuit}
      (f^*)}.
  \end{equation}
  Moreover, 
  \begin{equation}
    \label{eqn:ZpN+}
    \mathbb{E}\abs{\mathcal{Z}_{p,N}^{\diamondsuit} (f)} \leq
    \mathbb{E}\abs{\mathcal{Z}_{p,N}^{\diamondsuit} (f^*)}
  \end{equation}
\end{thm}

Theorem \ref{thm:Zp+} relies on first establishing the empirical
version \eqref{eqn:ZpN+}, while \eqref{eqn:Zp+} is derived as a
consequence. This is a key difference from the empirical approach in
\cite{PP12,CEFPP15,PP17} in which (non-random) inequalities of Lutwak,
Yang and Zhang \cite{LZ97,LYZ00,LYZ_Orlicz} inspired the development
of their empirical versions (e.g., \eqref{eqn:LZ} motivated its
stochastic form \eqref{eqn:CEFPP_dBP}). Recently, Yaskin proved
\eqref{eqn:Zp+} and extensions raised in \cite{KPZ11} in the case when
$f=\chi_{K}$, where $K$ is an origin-symmetric star body
sufficiently close to the Euclidean ball \cite{Yaskin22}.

Our original inspiration is a recent volume formula for sections of
finite-dimensional $L_p$ balls by Nayar and Tkocz \cite{NT20} that
builds on ideas involving Gaussian mixtures of random variables from
\cite{ENT18}.  Kindred probabilistic representations have been
indispensible in the study of sections of convex bodies, e.g.,
\cite{MeyPaj88,Kol98Israel,BGMN}. In our case, it allows for a
reduction from star-shaped sets to convex sets that interfaces well
with the empirical approach from \cite{PP12,CEFPP15,PP17}.

The methods we develop here go beyond centroid bodies, to families of
subspaces of $L_p$. For $f\in \mathcal{P}_n$, an origin-symmetric
convex body $C$ in $\mathbb{R}^m$, $m\geq 1$, and $p\not =0$, we
define $Z_{p,C}^{\diamondsuit}(f)\subseteq \mathbb{R}^n$ by its radial
function: for $p\not =0$,
\begin{equation}
  \rho^{-p}(Z_{p,C}^{\diamondsuit}(f),u) = 
  \int\limits_{(\mathbb{R}^n)^m}h^p(C,(\langle
  x_i,u\rangle)_{i=1}^m)
  \prod_{i=1}^{m}f(x_i)d\overline{x},
\end{equation}where $d\overline{x}= dx_1\ldots dx_m$, and for $p=0$, 
\begin{equation}
  \log \rho(Z_{0,C}^{\diamondsuit}(f),u) =
  -\int\limits_{(\mathbb{R}^n)^m}\log h(C,(\langle
  x_i,u\rangle)_{i=1}^m)\prod_{i=1}^mf(x_i) d\overline{x}.
\end{equation}

We also define empirical versions involving multiple bodies $C$ and
densities $f$. Specifically, let $C_1,\ldots,C_N$ be origin-symmetric
convex bodies with $m_i=\mathop{\rm dim}(C_i)\geq 1$ for $i\in
[N]=\{1,\ldots,N\}$, where $N>n$.  Let $(X_{ij})$, $i\in[N]$, $j\in[m_i]$ be
independent random vectors with $X_{ij}$ distributed according to
$f_{ij}\in \mathcal{P}_n$.  Write $\mathcal{C}=(C_1,\ldots,C_N)$ and
$\mathcal{F}=((f_{ij})_j)_i$.  For $p\not=0$, we define a star-shaped
set $Z_{p,\mathcal{C}}^{\diamondsuit}(\mathcal{F}) \subseteq \mathbb{R}^n$
by
\begin{eqnarray}
  \label{eqn:rhoZpCp} 
  \rho^{-p}(\mathcal{Z}^{\diamondsuit}_{p,\mathcal{C}}(\mathcal{F}),u)=
  \frac{1}{N}\sum_{i=1}^N h^p(C_i,(\langle X_{ij}, u
   \rangle)_{j=1}^{m_i});
\end{eqnarray}
for $p=0$, we define
$\mathcal{Z}_{0,\mathcal{C}}^{\diamondsuit}(\mathcal{F})\subseteq \mathbb{R}^n$ by its
radial function
\begin{equation}
  \label{eqn:rhoZpC0}
 \rho^{-N}(\mathcal{Z}_{0,\mathcal{C}}^{\diamondsuit}(\mathcal{F}),u)=
 \prod_{i=1}^N h(C_i,(\langle X_{ij}, u
\rangle)_{j=1}^{m_i}).
\end{equation}
For $p\geq 1$, the $\cdot^{\diamondsuit}$ notation agrees with usual
polarity.  When $p>0$ and $C=[-1,1]$, then
$Z^{\diamondsuit}_{p,[-1,1]}(f)=Z^{\diamondsuit}_{p}(f)$; similarly,
if $p>0$, $\mathcal{F}=(f)_{i=1}^N$ and
$\mathcal{C}=([-1,1])_{i=1}^N$, then
$\mathcal{Z}_{p,\mathcal{C}}^{\diamondsuit}(\mathcal{F})=
\mathcal{Z}_{p,N}^{\diamondsuit}(f)$.
For $p\geq 0$, we have the following generalization of Theorem
\ref{thm:Zp+}, going from $\mathcal{F}=(f_{ij})$ to the family of
rearranged densities $\mathcal{F}^{\#}=(f_{ij}^*)$.
 
\begin{thm}
\label{thm:ZpC+}
Let $f\in \mathcal{P}_n$ and let $p\geq 0$.  If $C$ is an
origin-symmetric convex body of dimension $m\geq 1$, then
\begin{equation}
  \label{eqn:ZpC+}
  \abs{Z_{p,C}^{\diamondsuit}(f)}\leq
  \abs{Z_{p,C}^{\diamondsuit}(f^*)}.
\end{equation}
Moreover, if $\mathcal{F}=(f_{ij})\subseteq \mathcal{P}_n$ and
$\mathcal{C}=(C_1,\ldots,C_N)$, where each $C_i$ is an
origin-symmetric convex body of dimension $m_i\geq 1$, then 
\begin{equation}
  \label{eqn:empZpC+}
\mathbb{E}\abs{\mathcal{Z}_{p,\mathcal{C}}^{\diamondsuit} (\mathcal{F})} \leq
\mathbb{E}\abs{\mathcal{Z}_{p,\mathcal{C}}^{\diamondsuit} (\mathcal{F}^{\#})}.
\end{equation}
\end{thm}

The theorem is new for all values of $p$. For $p\geq 1$, the proof
uses tools that have already been developed in \cite{CEFPP15}.
The main novelty here is in techniques to deal with the star-shaped
sets $\mathcal{Z}_{p,\mathcal{C}}^{\diamondsuit} (\mathcal{F})$ in the
range $p\in[0,1)$. In particular, we provide a separate treatment for
  $p=0$ including a new volume formula for
  $\mathcal{Z}_{0,\mathcal{C}}^{\diamondsuit}(\mathcal{F})$.  For
  $p<0$, the expected volume of the empirical bodies
  $\mathcal{Z}_{p,\mathcal{C}}^{\diamondsuit}(\mathcal{F})$ need not
  be finite when $\mathop{\rm dim}(C_i)<n$ (see Remark
  \ref{rem:infinite}); here the use of higher-dimensional convex
  bodies $C_1,\ldots, C_N$ is essential.  For certain values of $p$,
  namely when $p\in [-1,0)$ and $n/p$ is an integer, we establish the
    following theorem.

\begin{thm}
  \label{thm:ZpC-}
Let $f\in \mathcal{P}_n$ and let $p\in [-1,0)$. Let $C$ be an
  origin-symmetric convex body with $\mathop{\rm dim}(C)\geq 1$. If
  $p>-1$ and $n/\abs{p}\in \mathbb{N}$, then
  \begin{equation}
    \label{eqn:ZpC-}
    \abs{Z_{p,C}^{\diamondsuit}(f)} \leq \abs{Z_{p,C}^{\diamondsuit} (f^*)}.
  \end{equation}
Furthermore, let $\mathcal{F}=(f_{ij})\subseteq \mathcal{P}_n$ and
$\mathcal{C}=(C_1,\ldots,C_N)$, where each $C_i$ is an
origin-symmetric convex body of dimension $m_i\geq n+1$. If 
$p\geq -1$ and $n/\abs{p}\in \mathbb{N}$, then
\begin{equation}
  \label{eqn:empZpC-}
\mathbb{E}\abs{\mathcal{Z}_{p,\mathcal{C}}^{\diamondsuit} (\mathcal{F})} \leq
\mathbb{E}\abs{\mathcal{Z}_{p,\mathcal{C}}^{\diamondsuit} (\mathcal{F}^{\#})}.
\end{equation}
\end{thm}

Empirical versions of isoperimetric inequalities from
\cite{PP12,CEFPP15,PP17} have involved operations in Brunn-Minkowski
theory; e.g., for $p\geq 1$, the sets $\mathcal{Z}_{p,\mathcal{C}}(f)$
in \eqref{eqn:Zp1} are $L_p$ sums of random line segments (see \S
\ref{section:dual_centroid}). Theorems \ref{thm:Zp+} - \ref{thm:ZpC-}
are the first to treat empirical forms of inequalities for star-shaped
sets in dual Brunn-Minkowski theory. In particular, we develop
randomized analogues of approximation results of Goodey and Weil
\cite{GooWei95}, and Kalton, Koldobsky, Yaskin and Yaskina
\cite{KKYY07}, in which intersection bodies and their $L_p$ analogues
are limits of radial sums of ellipsoids. The use of higher-dimensional
bodies $C_i$ in Theorem \ref{thm:ZpC-} is needed for this purpose and
such bodies are crucial for establishing the corresponding
isoperimetric inequalities.  In particular, we define a variant of the
$L_p$-intersection body as follows: for $f\in \mathcal{P}_n$,
$\alpha>0$ and $p\in[-1,0)$, we set
\begin{equation*}
  \rho^{\abs{p}}(I_{\abs{p}}^{\alpha}(f), u) =
  \int_{\mathbb{R}^n}\left(\absinprod{x}{u}^{2}+\alpha^2\norm{u}_2^2\right)^{-\abs{p}/2}f(x)dx.
\end{equation*}
For the empirical version, we consider $N>n$ independent random
vectors $X_1,\ldots,X_N$ from $f\in \mathcal{P}_n$ and define
$\mathcal{I}_{\abs{p},N}^{\alpha}(f)$ via
\begin{equation*}
  \rho^{\abs{p}}(\mathcal{I}_{\abs{p},N}^{\alpha}(f), u) =
  \frac{1}{N} \sum_{i=1}^N
  \left(\absinprod{X_i}{u}^{2}+\alpha^2\norm{u}_2^2\right)^{-\abs{p}/2}.
\end{equation*}
The star-shaped bodies $\mathcal{I}_{\abs{p},N}^{\alpha}(f)$ are
$L_p$-radial sums of ellipsoids (see \S \ref{section:prelim} for
definitions). In fact, the bodies
$\mathcal{I}_{\abs{p},N}^{\alpha}(f)$ are special (limiting) cases of
$\mathcal{Z}_{p,\mathcal{C}}^{\diamondsuit}(\mathcal{F})$ for a
suitable choice of $\mathcal{C}$ and $\mathcal{F}$, involving
ellipsoids and uniform measures on balls. 

\begin{cor}
  \label{cor:int}
  Let $f\in \mathcal{P}_n$, $\alpha >0$, $p\in[-1,0)$ and
    $n/\abs{p}\in \mathbb{N}$. Then
  \begin{equation}
    \label{eqn:int}
    \abs{I_{\abs{p}}^{\alpha}(f)}\leq \abs{I_{\abs{p}}^{\alpha}(f^*)}.
  \end{equation}
  Moreover, 
  \begin{equation}
    \label{eqn:emp_int}
    \mathbb{E} \abs{\mathcal{I}_{\abs{p},N}^{\alpha}(f)}\leq
    \mathbb{E} \abs{\mathcal{I}_{\abs{p},N}^{\alpha}(f^*)}.
  \end{equation}
\end{cor}

When $p=-1$, \eqref{eqn:emp_int} is a stochastic form of the Busemann
intersection inequality \eqref{eqn:BusInt}, as it implies the latter
when $N\rightarrow \infty$ and $\alpha \rightarrow 0$. Indeed, if
$f\in \mathcal{P}_n$, we write $I(f)$ for the intersection body of
$f$, defined by its radial function via
\begin{equation*}
  \rho(I(f),u)=\int_{u^{\perp}}f(x)dx,
\end{equation*}
and \eqref{eqn:emp_int} implies the following functional version
of \eqref{eqn:BusInt}.

\begin{cor}
  \label{cor:If}
Let $f$ be a continuous and compactly supported function in
$\mathcal{P}_n$. Then
\begin{equation}
  \label{eqn:If}
    \abs{I(f)}\leq \abs{I(f^*)}.
  \end{equation}
\end{cor}

Thus the Busemann intersection inequality \eqref{eqn:BusInt} is one
limiting case of a family of extremal inequalities about $L_p$-radial
sums in Theorem \ref{thm:ZpC-}. For (non-random) functional versions
of the Busemann intersection inequality, see \cite{DPP16}, and
\cite{Haddad21} for recent developments.

Lastly, we can further reduce inequalities to uniform measures on
balls in each of the above theorems whenever the convex bodies $C$ and
$C_i$ are unconditional, i.e., invariant under reflections in the
coordinate hyperplanes.

\begin{thm}
  \label{thm:ball}
Let $f\in \mathcal{P}_n$. Suppose that $p\in [0,1]$, or $p\in[-1,0)$
  and $n/\abs{p}\in \mathbb{N}$.  Let $C$ be an unconditional convex
  body in $\mathbb{R}^m$, $m\geq 1$.  Set
  $g=\norm{f}_{\infty}\chi_{r B_2^n}$, where $r>0$ satisfies $\int
  g =1$. Then for $p>-1$, 
  \begin{equation*}
    \abs{Z_{p,C}^{\diamondsuit}(f)}\leq
    \abs{Z_{p,C}^{\diamondsuit}(g)},
  \end{equation*}while for $p\geq -1$, 
  \begin{equation*}
    \mathbb{E} \abs{\mathcal{I}_{\abs{p},N}^{\alpha}(f)}\leq
    \mathbb{E} \abs{\mathcal{I}_{\abs{p},N}^{\alpha}(g)}.
    \end{equation*}
  Furthermore, assume that $\mathcal{F}=(f_{ij})\subseteq
  \mathcal{P}_n$ and $\mathcal{G}=(g_{ij})$, where $g_{ij} =
  \norm{f_{ij}}_{\infty}\chi_{r_{{ij}}B_2^n}$ with $r_{ij}>0$
  satisfying $\int g_{ij}=1$. If $\mathcal{C}=(C_1,\ldots,C_N)$, where
  each $C_i$ is an unconditional convex body of dimension $m_i$ as above, then
  \begin{equation*}
    \mathbb{E}\abs{\mathcal{Z}_{p,\mathcal{C}}^{\diamondsuit}(\mathcal{F})}\leq
    \mathbb{E}\abs{\mathcal{Z}_{p,\mathcal{C}}^{\diamondsuit}(\mathcal{G})}.
  \end{equation*}
\end{thm}

The paper is organized as follows: \S \ref{section:prelim} introduces
notation and basic tools; \S \ref{section:dual_centroid} is devoted to
the non-random bodies $Z_{p,C}^{\diamondsuit}(f)$ and variants of
$L_p$-intersection bodies; \S \ref{section:emp_dual_centroid} develops
the randomized versions of these objects. New volume formulas and
representations for radial functions are developed in \S
\ref{section:volume_formulas}. The theorems are proved in \S
\ref{section:main}.

\section{Preliminaries}
\label{section:prelim}

\subsection{Notation and definitions}

For a compact set $K\subseteq \mathbb{R}^n$, we denote its convex hull
by $\mathop{\rm conv}(K)$.  The set of all compact, convex sets in
$\mathbb{R}^n$ will be denoted by $\mathcal{K}^n$. For $K\in
\mathcal{K}^n$, its support function is defined by $h(K,u)=\sup_{x\in
  K}\langle x, u\rangle$, $u\in \mathbb{R}^n$.  The Hausdorff metric
on $\mathcal{K}^n$ is defined by
\begin{equation*}
  \delta^H(K,L)=\sup_{\theta\in S^{n-1}}\abs{h(K,\theta)-h(L,\theta)},
\end{equation*}
where $S^{n-1}$ is the unit sphere. We call $K\in \mathcal{K}^n$ a
convex body if it has interior points. We say that $K\in
\mathcal{K}^n$ is origin-symmetric if $-x\in K$ whenever $x\in K$. The
set of all origin-symmetric convex bodies in $\mathbb{R}^n$ will be
denoted by $\mathcal{K}_s^n$. Each $K\in \mathcal{K}_s^n$ gives rise
to a norm on $\mathbb{R}^n$ given by
\begin{equation*}
  \norm{u}_{K}=\inf\{\lambda >0: u\in \lambda K\}.
\end{equation*}
The polar body of $K\in \mathcal{K}_s^n$ is defined by
$K^{\circ}=\{u\in K:h_K(u)\leq 1\}$.

For measurable sets $A\subseteq \mathbb{R}^n$, we use $\abs{A}$ for
the Lebesgue measure of $A$. By $\omega_n$, we mean the volume of the
Euclidean ball in $\real^n$ with radius 1, i.e.,
$\omega_n=\pi^{n/2}/\Gamma(n/2+1)$.

We will call a set $K$ in $\mathbb{R}^n$ star-shaped if $0\in K$ and
$\alpha x\in K$ whenever $x\in K$ and $\alpha \in [0,1]$.  The radial
function of a star-shaped set $K$ is defined as
$\rho(K,u)=\sup\{\alpha\geq 0:\alpha u \in K\}$ for $u\in S^{n-1}$.
Here we allow $K$ to be unbounded and $\rho(K,u)$ may take the value
$+\infty$. As our focus is volumetric inequalities, we are
particularly interested in radial functions of star-shaped sets $K$
with $\rho(K,\cdot)\in L_{n}(S^{n-1},\sigma)$ in which case we write
\begin{equation*}
  \norm{\rho(K,\cdot)}_n=
  \left(\int_{S^{n-1}}\rho^n(K,u)du\right)^{1/n} = \omega_n^{-1/n}\abs{K}^{1/n}.
\end{equation*}
Throughout, $du$ denotes $d\sigma(u)$, where $\sigma$ is
the normalized Haar probability measure on $S^{n-1}$.

We will call $K$ a star-body if it is a compact, star-shaped set with
the origin in its interior and its radial function is continuous. When
$K\in \mathcal{K}_s^n$, we have for $u\in \mathbb{R}^n\backslash
\{0\}$,
\begin{equation*}
  \rho(K,u)= \norm{u}^{-1}_K \quad \text{ and }\quad \rho(K^{\circ},u)=h_K^{-1}(u).
\end{equation*}

We recall a core notion of addition of convex bodies from $L_p$
Brunn-Minkowski theory, e.g. \cite{Firey,Lutwak93JDG,Lutwak96}. For
$K,L\in \mathcal{K}^n$ containing the origin and $p\geq 1$, we will
write $K+_{p}L$ for their $L_p$ sum, i.e.,
\begin{equation*}
  h^p(K+_pL,u)=h^p(K,u)+h^p(L,u)\quad (u\in \mathbb{R}^n).
\end{equation*}
In dual Brunn-Minkowski theory, (e.g., \cite{LZ97, Schneider_book}), for
star-bodies $K,L$, and $p\not=0$, their $L_p$-radial sum $K\tilde{+}_p
L$ is defined by
\begin{equation*}
  \rho^p(K\tilde{+}_pL,u)=\rho^p(K,u)+\rho^p(L,u) \quad (u\in S^{n-1}).
\end{equation*}

For a measurable set $A$ in $\real^n$ with finite volume, we define
its rearrangement $A^*$ to be the (open) Euclidean ball centered at
the origin satisfying $|A^*| = |A|$.  We will use the following
bracket notation for indicator functions:
\begin{equation}
  \label{eqn:bracket}
  [u\in A ] = \chi_A(u).
\end{equation}
For a non-negative integrable function $f$ on $\mathbb{R}^n$, its
layer-cake representation is given by
\begin{equation}
  \label{eqn:layer_cake}
f(x)=\int_{0}^{\infty}\chi_{\{f>t\}}(x)dt = \int_{0}^{\infty}[x\in \{f>t\}]dt.
\end{equation}
The symmetric decreasing rearrangement of a non-negative integrable
function $f$ on $\mathbb{R}^n$ is defined using rearrangement of its
level sets $\{x\in \mathbb{R}^n:f(x)>t\}=\{f>t\}$ ($t>0$) via
\begin{equation}
f^* (x) = \int_0^\infty \chi_{\{f > t\}^*}(x) dt =
\int_{0}^{\infty}[x\in \{f>t\}^*]dt.
\end{equation}
For a general reference on rearrangements, we refer the reader to
\cite{LiebLoss}.  We will use the fact that $f$ and $f^*$ are
equimeasurable; in particular, $f^*$ preserves all $L_p$ norms of
$f$. Note also that if $f\leq g$, then $f^*\leq g^*$.  Moreover,
rearrangements satisfy the following contractive property: for $1\leq
p\leq \infty$ and for $f,g\in L_p$,
 \begin{equation}
   \label{eqn:contraction}
   \norm{f^*-g^*}_p\leq \norm{f-g}_p.
 \end{equation}

For $f\in \mathcal{P}_n$, the marginal density of $f$ on a subspace
$E$ of dimension $k$, is defined as
\begin{equation}
  \label{eqn:marginal}
  \pi_E(f)(x)=\int_{E^{\perp}+x} f(y) dy,
\end{equation}
where $E^{\perp}$ denotes the orthogonal complement of $E$. Note that
when $f\in \mathcal{P}_n$ and has compact support, then $\pi_E(f)$ is
also bounded and has compact support.

\subsection{Probabilistic tools}

We will make repeated use of the following fact about uniformly
integrable collections of random variables (e.g.,
\cite[pg. 189]{Shir_book}). 

\begin{prop}
  \label{prop:UI}
Let $\eta, \eta_1,\eta_2,\ldots$ be non-negative random variables on a
probability space $(\Omega, \mathcal{M},\mathbb{P})$ such that
$\eta_k\rightarrow \eta$ as $k\rightarrow \infty$ almost surely. If
$\{\eta_k\}$ is uniformly integrable, then \begin{equation*}
  \lim\limits_{k\rightarrow \infty}\mathbb{E}\eta_k = \mathbb{E} \eta
  <\infty.
  \end{equation*}
\end{prop}

\begin{remark}
  \label{remark:UI}
  A sufficient condition for uniform integrability of a family of
  random variables $\{\eta_k\}$ is boundedness in
  $L_{1+\delta}(\Omega, \mathcal{M}, \mathbb{P})$, for some $\delta>0$
  (\cite[pg. 190]{Shir_book}).
\end{remark}

We will also use Kolmogorov's strong law of large numbers
(\cite[pg. 391]{Shir_book}).

\begin{prop}
  \label{prop:SLLN}
  Let $\eta_1,\eta_2,\ldots$ be independent identically distributed
  random variables on a probability space $(\Omega,
  \mathcal{M},\mathbb{P})$ such that
  $\mathbb{E}\abs{\eta_1}<\infty$. Then, almost surely, as
  $N\rightarrow \infty$,
  \begin{equation*}
    \frac{1}{N}\sum_{k=1}^N\eta_k \rightarrow \mathbb{E}\eta_1.
    \end{equation*}
\end{prop}

We will frequently use {\it a.s.} as an abbreviation for almost sure
convergence; similarly, we use {\it i.i.d.} for a sequence of
independent identically distributed random variables.

\subsection{Volume in terms of Gaussian integrals}
 
We will use the following elementary lemma which relates the volume of
star-shaped sets to certain Gaussian integrals.

\begin{lemma}
  \label{lemma:Gaussian_vol}
   Let $K$ be a star-shaped set with $0\in\mathop{\rm int}(K)$ and
   $\rho(K,\cdot) \in L_n(S^{n-1},\sigma)$.  If $\xi$ is a standard
   Gaussian vector in $\mathbb{R}^n$, and $s\in(0,n)$, then
  \begin{equation}
    \label{eqn:sphere_Gauss}
    \mathbb{E}_{\xi}\rho^s(K,\xi)= b_{n,s}\int_{S^{n-1}} \rho^s(K,u)du,
  \end{equation}
where
\begin{equation}
  \label{eqn:bns}
b_{n,s}=\mathbb{E}_{\xi}\norm{\xi}_2^{-s}=
\frac{n\Gamma(\frac{n-s}{2})}{2^{s/2+1}\Gamma(\frac{n}{2}+1)}.
\end{equation}
Furthermore, if $\rho(K,\cdot)$ is additionally the pointwise limit of
an increasing sequence of radial functions $\{\rho(K_{\ell}, \cdot)\}$ of
star shaped sets $\{K_{\ell}\}$, then
\begin{equation}
  \label{eqn:lim}
\abs{K} = \lim_{\ell\rightarrow
  \infty}\frac{\mathbb{E}_{\xi}\rho^{n-1/\ell}(K_{\ell},\xi)}{b_{n,n-1/\ell}}.
\end{equation}
\end{lemma}

\begin{proof}
Using polar coordinates, we have for $0<s<n$, 
  \begin{eqnarray*}
    \mathbb{E}_{\xi} \rho^s(K,\xi) & = &
    \frac{n\omega_n}{(2\pi)^{n/2}}\int_{0}^{\infty}r^{n-s-1}e^{-r^2/2}dr\int_{S^{n-1}}
    \rho^s(K,u)du\\ & = &
    \frac{n\Gamma(\frac{n-s}{2})}{2^{s/2+1}\Gamma(\frac{n}{2}+1)}\int_{S^{n-1}}
    \rho^s(K,u)du.
  \end{eqnarray*}
  The conditions $0\in \mathop{\rm int}(K)$ and $\rho(K,\cdot)\in
  L_n(S^{n-1},\sigma)$ ensure that $\rho(K,u)$ is positive and finite
  for all $u$ outside of a null set on $S^{n-1}$. For such $u$, since
  $\rho(K_{\ell},u)\rightarrow \rho(K,u)$, we have
  $$\rho^{n-1/\ell}(K_{\ell},u) = \rho^n(K_{\ell},u)\exp\left(-\frac{\log
    \rho(K_{\ell},u)}{\ell}\right)\rightarrow \rho^n(K,u).$$ Next, since
  $\{\rho(K_{\ell},u)\}$ is increasing,
  \begin{equation}
    \label{eqn:mon_dom}
    \rho^{n-1/\ell}(K_{\ell},u) \leq  \max(1, \rho^{n}(K_{\ell},u))
    \leq  \max(1, \rho^{n}(K,u))
    \leq 1+\rho^n(K,u).
  \end{equation}
  By dominated convergence and \eqref{eqn:sphere_Gauss}, we get
  \begin{eqnarray*}
    \omega_n^{-1}\abs{K} = \int_{S^{n-1}}\rho^n(K,u)du = \lim_{\ell\rightarrow \infty}
    \int_{S^{n-1}}\rho^{n-1/\ell}(K_{\ell},u)du = \lim_{\ell\rightarrow
      \infty}
    \frac{\mathbb{E}_{\xi}\rho^{n-1/\ell}(K_{\ell},\xi)}{b_{n,n-1/\ell}}.
  \end{eqnarray*}
\end{proof}

\section{Dual $L_{p,C}$-centroid bodies}

\label{section:dual_centroid}

Let $f\in \mathcal{P}_n$, $p>-1$ and let $C$ be an origin-symmetric
convex body in $\mathbb{R}^m$, $m\geq 1$. For ease of reference, we recall
that for $p\not =0$,
\begin{equation}
  \rho^{-p}(Z_{p,C}^{\diamondsuit}(f),u) = 
  \int\limits_{(\mathbb{R}^n)^m}h^p(C,(\langle
  x_i,u\rangle)_{i=1}^m)
  \prod_{i=1}^{m}f(x_i)d\overline{x}
\end{equation}and for $p=0$, 
\begin{equation}
  \log \rho(Z_{0,C}^{\diamondsuit}(f),u) = -\int\limits_{(\mathbb{R}^n)^m}\log
h_{C}((\langle
  x_i,u\rangle)_{i=1}^m)\prod_{i=1}^mf(x_i)
d\overline{x}.
\end{equation}
As noted in the introduction, the latter bodies are not convex in
general. We will use the term {\it dual $L_{p,C}$ centroid body} as
these bodies fit within dual Brunn-Minkowski theory. This agrees with
the convex case when $p\geq 1$, however, the term here is meant in a
broader sense than duality for convex bodies.  When $p\geq 1$,
$Z_{p,N}^{\diamondsuit}(f)=Z_{p,N}^{\circ}(f)$.

We start by noting a few elementary properties of the bodies
$Z_{p,C}^{\diamondsuit}(f)$. 

\begin{lemma}
  \label{lemma:Zbasic}
  Let $f\in \mathcal{P}_n$, $p,p_1,p_2>-1$ and $C\in \mathcal{K}_s^m$,
  $m\geq 1$.
\begin{itemize}
\item[(a)]If $p_1\leq p_2$, then
  \begin{equation*}
    Z_{p_2,C}^{\diamondsuit}(f) \subseteq Z_{p_1,C}^{\diamondsuit}(f).
  \end{equation*}
\item[(b)]If $D\in \mathcal{K}_s^{m_1}$, and $C\subseteq D$, then
  \begin{equation*}
    Z_{p,C}^{\diamondsuit}(f) \supseteq Z_{p,D}^{\diamondsuit}(f).
  \end{equation*}
\item[(c)] $\rho^{\abs{p}}(Z_{p,C}^{\diamondsuit}(f),\cdot)\in
L_1(S^{n-1},\sigma)$. 
\item[(d)] For $k\in \mathbb{N}$ such that $\int_{kB_2^n}f(x)dx >0$,
  let $\varphi^{(k)}=f|_{kB_2^n}$ and $\phi^{(k)}=\varphi^{(k)}/\int
  \varphi^{(k)}$. Then
  \begin{equation*}\rho(Z_{p,C}^{\diamondsuit}(f),u)=
    \lim_{k\rightarrow \infty}\rho(Z_{p,C}^{\diamondsuit}(\phi^{(k)}),u) 
    \quad (u\in S^{n-1}).
    \end{equation*}
\end{itemize}
\end{lemma}

\begin{proof}
Part (a) is a consequence of H\"{o}lder's inequality.  For (b), the
condition $C\subseteq D$ is equivalent to $h(C,\cdot)\leq
h(D,\cdot)$, hence 
$\rho(Z_{p,D}^{\diamondsuit}(f),u)\leq
\rho(Z_{p,C}^{\diamondsuit}(f),u)$ for each $u\in S^{n-1}$.

By using (a), it is sufficient to treat (c) for $p\in(-1,0)$. Since
$C\in \mathcal{K}_s^m$, we can assume $C\subseteq \mathop{\rm
  span}\{e_1,\ldots,e_m\}$, and there exists $r_0>0$ such that
$r_0[-e_1,e_1]\subseteq C$, hence $\rho(Z_{p,C}^{\diamondsuit}(f),u)
\leq r_0^{-1}\rho(Z_p^{\diamondsuit}(f),u)$ for $u\in S^{n-1}$.  For
$p\in (-1,0)$, we have for each $u\in S^{n-1}$,
\begin{equation*}
  \norm{x}_2^{-\abs{p}} = \beta_{n,p}
  \int_{S^{n-1}}\absinprod{x}{u}^{-\abs{p}}du,
\end{equation*}
where $\beta_{n,p}=b_{n,\abs{p}}/b_{1,\abs{p}}$ (cf. \eqref{eqn:bns}).
Since $x\mapsto \norm{x}_2^{-\abs{p}}$ is locally integrable and $f\in
\mathcal{P}_n$, we have
\begin{equation*}
  \int_{\mathbb{R}^n}\norm{x}_2^{-\abs{p}}f(x)dx \leq \norm{f}_{\infty}
  \int_{B_2^n}\norm{x}_2^{-\abs{p}}dx + \int_{\mathbb{R}^n\backslash
    B_2^n}f(x) dx < \infty.
\end{equation*}Thus part (c) follows from
\begin{eqnarray*}
\int_{S^{n-1}}\rho^{\abs{p}}(Z_{p}^{\diamondsuit}(f),u)du 
&=&\int_{S^{n-1}}\int_{\mathbb{R}^n}\absinprod{x}{u}^{-\abs{p}}f(x)dx du\\
& = & \beta_{n,p}^{-1}\int_{\mathbb{R}^n}\norm{x}_2^{-\abs{p}}f(x)dx.
\end{eqnarray*}

To prove (d), we note that part (c) implies
$\rho(Z_{p,C}^{\diamondsuit}(f), u)<\infty$ for a.e. $u\in
S^{n-1}$. Since $\varphi^{(k)}\rightarrow f$, and $f\in
\mathcal{P}_n$, we have $\int \varphi^{(k)}\rightarrow \int f=1$. For
$p\not =0$, (d) follows by monotone convergence:
\begin{equation}
  \label{eqn:mon_lim}
  \int_{(\mathbb{R}^n)^m} h^p(C,(\langle x_i,u
  \rangle)_{i=1}^m)\prod_{i=1}^m \varphi^{(k)}(x_i)d\overline{x} \rightarrow 
  \int_{(\mathbb{R}^n)^m} h^p(C,(\langle x_i,u
  \rangle)_{i=1}^m)\prod_{i=1}^m f(x_i)d\overline{x};
\end{equation}
the latter holds even when the righthand side of \eqref{eqn:mon_lim}
is infinite, which may occur for $p>0$, in which case
$\rho(Z_{p,C}^{\diamondsuit}(f),u)=0$.  To treat $p=0$, we set
\begin{equation*}
P_1(u)=\{(x_i)_{i=1}^m\in (\mathbb{R}^n)^m: h(C,(\langle x_i,u
\rangle)_{i=1}^m)>1\}
\end{equation*} and $P_2(u)=(\mathbb{R}^n)^m\backslash P_1(u)$.
Applying the same argument to each factor in \begin{eqnarray*}
  \rho(Z_{0,C}^{\diamondsuit}(\phi^{(k)}),u) = \prod_{i=1}^2
  \exp\left(-\int_{P_i(u)}\log h(C,(\langle x_i,u
  \rangle)_{i=1}^m\prod_{i=1}^m \phi^{(k)}(x_i)d\overline{x}\right),
  \end{eqnarray*}
shows that (d) holds for $p=0$ as well.
\end{proof}

\subsection{$L_p^{\alpha}$-intersection bodies}

For $f\in \mathcal{P}_n$, we write $I(f)$ for its intersection body, 
defined by its radial function via
\begin{equation}
  \label{eqn:int_f}
  \rho(I(f),u)=\int_{u^{\perp}}f(x)dx,
  \end{equation}
where the integration is with respect to Lebesgue measure on
$u^{\perp}$; for background on intersection bodies, see
\cite{Lutwak88,Kol05,Gar06}.  Motivated by approximation results for
intersection bodies involving radial sums of ellipsoids
\cite{GooWei95,KKYY07}, we define a variant of \eqref{eqn:int_f}: for
$\alpha>0$ and $p\in [-1,0)$, the $L_{p}^{\alpha}$-intersection body
  of $f$ is given by
\begin{equation*}
  \rho^{\abs{p}}(I_{\abs{p}}^{\alpha}(f), u) = \int_{\mathbb{R}^n}
  \left(\absinprod{x}{u}^2+\alpha^2\norm{u}_2^2\right)^{-\abs{p}/2}f(x)dx.
\end{equation*}
As mentioned, when $f$ is the indicator of a star-body and $\alpha=0$,
the latter bodies were studied in \cite{YY06,HL06,Haberl08}.  When
$p=-1$ and $\alpha>0$, we write $I^{\alpha}(f)=I_1^{\alpha}(f)$.

\begin{prop}
  \label{prop:asinh}
  Let $f$ be a continuous, compactly supported function in
  $\mathcal{P}_n$.  For $\alpha>0$, let
  $s_\alpha=\sinh^{-1}(1/\alpha)$. Then
  \begin{equation*}
    \abs{I(f)}=\lim_{\alpha \rightarrow 0}(2s_{\alpha})^{-n}\abs{
      I^{\alpha}(f)}.
  \end{equation*}
\end{prop}
We will prove this using an approximate identity, i.e., a family of
non-negative functions $(k_{\alpha})_{\alpha\in(0,1)}$ on $\mathbb{R}$
satisfying the following conditions, for each $\alpha\in (0,1)$,
\begin{itemize}
\item[(i)] $\int_{\mathbb{R}}k_{\alpha}(t)dt=1$;
\item[(ii)]for any $\delta >0$, 
  $\lim_{\alpha\rightarrow 0}\int_{\abs{t}>\delta}k_{\alpha}(t)dt =0$.
\end{itemize}
In this case, if $g$ is continuous and supported on a
compact set $K$, then $\norm{(k_{\alpha}*g)-g}_{L_{\infty}(K)}\rightarrow
0$ (see, e.g., \cite[pg. 27]{Graf14}).
\begin{proof}
For $\alpha>0$, let
  \begin{equation*}
    k_{\alpha}(t) =
    (2s_{\alpha})^{-1}\left(t^2+\alpha^2\right)^{-1/2}\chi_{[-1,1]}(t).
  \end{equation*}  
  Standard computations show that $(k_{\alpha})_{\alpha}$ is an
  approximate identity.  Fix $u\in S^{n-1}$ and recall the notation
  for the marginal of $f$ on $[u]=\mathop{\rm span}\{u\}$ (cf
  \eqref{eqn:marginal}), and set $f_u(t) =\pi_{[u]}(f)(t)$. Then
  \begin{eqnarray}
    (2s_{\alpha})^{-1}\rho(I^{\alpha}(f),u) &=&
    (2s_{\alpha})^{-1}\int_{\mathbb{R}}\left(t^2+\alpha^2\right)^{-1/2}f_u(t)dt\nonumber\\ & = &
    \int_{\abs{t}\leq 1}k_{\alpha}(t)f_u(t)dt +
    (2s_{\alpha})^{-1}\int_{\abs{t}>1}\left(t^2+\alpha^2\right)^{-1/2}f_u(t)dt \label{eqn:k_sum}\\
&=&(k_{\alpha}*f_u)(0)+
    (2s_{\alpha})^{-1}\int_{\abs{t}>1}\left(t^2+\alpha^2\right)^{-1/2}f_u(t)dt. \nonumber
  \end{eqnarray}
  We have $k_{\alpha}*f_u(0)\rightarrow f_u(0)=\rho(I(f),u)$.  Since
  $\int_{\mathbb{R}}f_u(t)dt=1$ and $s_{\alpha}\rightarrow \infty$ as
  $\alpha\rightarrow 0$, we have
  \begin{equation*}
    \lim_{\alpha \rightarrow
      0}(2s_{\alpha})^{-1}\int_{\abs{t}>1}\left(t^2+\alpha^2\right)^{-1/2}f_u(t)dt
    =0.
  \end{equation*}
  It follows that
  \begin{equation*}(2s_{\alpha})^{-n}\rho^n(I^{\alpha}(f),u) \rightarrow
  \rho^n(I(f),u).
  \end{equation*}
Using \eqref{eqn:k_sum}, the latter convergence is dominated on
$(S^{n-1},\sigma)$ by $(\sup_{u}\norm{f_u}_{\infty}+(2s_1)^{-1})^{n}$,
hence
  \begin{equation*}
    \abs{I(f)} = \omega_n \int_{S^{n-1}}\lim_{\alpha \rightarrow
      0}\rho^n((2s_{\alpha})^{-1}I^{\alpha}(f),u)du
     =  \lim_{\alpha\rightarrow 0} (2s_{\alpha})^{-n}\abs{I^{\alpha}(f)}.
  \end{equation*}
\end{proof}

\section{Empirical dual $L_{p,C}$-centroid bodies}

\label{section:emp_dual_centroid}

An empirical approach to $L_p$-centroid bodies was initiated in
\cite{PP12} and developed further in \cite{CEFPP15,PP17}.  It relies
on random linear operators acting on various sets in
finite-dimensional normed spaces.  In this section, we recall the main
theorem from \cite{CEFPP15}.  We lay the groundwork to re-interpret
the random star-shaped bodies
$\mathcal{Z}_{p,\mathcal{C}}^{\diamondsuit}(\mathcal{F})$ of our main
theorems as random sections of $\ell_p$-balls. We also develop new
notions of randomly generated intersection bodies.

\subsection{Tools from the empirical approach}

It will be useful to fix some notation for matrices acting as linear
operators. For an $n\times N$ matrix ${\bf X}=[x_1\ldots x_N]$, we
write ${\bf X}^T$ for the transpose of ${\bf X}$ and we view ${\bf
  X}:\real^N\rightarrow \real^n$ and ${\bf X}^T:\real^n\rightarrow
\real^N$ as linear operators. In particular, for an origin-symmetric
convex body $C\subseteq \mathbb{R}^N$,
\begin{equation*}
  {\bf X}C=\left\{{\bf X}c: c\in C\right\}=\left\{\sum_{i=1}^N
  c_ix_i:c=(c_i)\in C\right\}.
\end{equation*}
Principal examples include (i) $C=B_1^N=\mathop{\rm conv}\{\pm
e_1,\ldots,\pm e_N\}$ and (ii)  $C=B_{\infty}^N=[-1,1]^N$  in which case
\begin{equation*}
  \text{(i)}{\bf X}B_1^N= \mathop{\rm conv}\{\pm X_1,\ldots,\pm X_N\} \quad
       \text{(ii)}{\bf X}B_{\infty}^N = \sum_{i=1}^{N}[-X_i,X_i].
\end{equation*}
Volumetric inequalities for convex hulls of random points and random
zonotopes \cite{Groemer,BMMP} motivated work in \cite{PP12} to
interpolate between these two extremes and led to an empirical study
of $L_p$ centroid bodies; see the survey \cite{PP17} and the
references therein.

All of the theorems in the introduction will be derived from the
following result about polars of convex bodies from \cite{CEFPP15}.
It concerns radial measures with decreasing densities (``decreasing''
is meant in a non-strict sense).

\begin{thm}
  \label{thm:CEFPP}
  Let ${\bf X}$ and ${\bf X}^{\#}$ be $n\times N$ random matrices with
  independent columns drawn from $\mathcal{F}= (f_i)_{i=1}^N\subseteq
  \mathcal{P}_n$ and $\mathcal{F}^{\#}=(f_i^*)_{i=1}^N$,
  respectively. Let $\nu$ be a radial measure with a decreasing
  density, i.e., $d\nu(x)=h(\norm{x}_2)dx$ with
  $h:[0,\infty)\rightarrow [0,\infty)$ decreasing. Then for any
      origin-symmetric convex body $C$ in $\mathbb{R}^N$,
  \begin{equation}
    \label{eqn:CEFPPa}
    \mathbb{E}\nu\left(({\bf X} C)^{\circ}\right) \leq \mathbb{E}\nu\left(({\bf
      X}^{\#} C)^{\circ}\right).
  \end{equation} 
  Assume additionally that each $f_i$ is bounded and ${\bf Z}$ is an
  $n\times N$ random matrix with independent columns drawn from
  $g_{i}=\norm{f_i}_{\infty}\chi_{r_i B_2^n}$, where $r_{i}>0$
  satisfies $\int g_{i}=1$. Then for any unconditional convex body $C$
  in $\mathbb{R}^N$,
  \begin{equation}
    \label{eqn:CEFPPb}
    \mathbb{E}\nu(({\bf X} C)^{\circ}) \leq \mathbb{E}\nu(({\bf Z}
    C)^{\circ}).
  \end{equation} 
\end{thm}

The latter theorem relies on rearrangement inequalities of Rogers
\cite{Rogers_single}, Brascamp-Lieb-Luttinger \cite{BLL} and Christ
\cite{Christ84}.  It also relies on the Borell-Brascamp-Lieb
inequalities \cite{Bor75,BL76}.  It was motivated by work of Campi and
Gronchi on symmetrization of polar convex bodies \cite{CG06}.

The following lemma is a useful re-interpretation of the bodies $({\bf
  X}C)^{\circ}$, stated in terms of the transpose ${\bf X}^T$ and its
pre-image ${\bf X}^{-T}$.

\begin{lemma} 
  \label{lemma:dual}
Let ${\bf X}$ be an $n\times N$ matrix of full rank, viewed as a
linear operator ${\bf X}:\mathbb{R}^N\rightarrow \mathbb{R}^n$.  Then
for $C\in \mathcal{K}_s^N$,
\begin{equation}
  \label{eqn:dual}
  ({\bf X}C)^{\circ} = {\bf X}^{-T}[C^{\circ}]. 
\end{equation}
\end{lemma}

\begin{proof}
When $N=n$, then ${\bf X}$ and ${\bf X}^T$ are invertible and
\eqref{eqn:dual} is a standard identity:
\begin{eqnarray*}
  ({\bf X}C)^{\circ} &=&\left\{x\in \mathbb{R}^n: \langle x, {\bf X}c
  \rangle \leq 1 \text{ for all } c\in C\right\}\\ &=& \left\{x\in
  \mathbb{R}^n: \langle {\bf X}^Tx, c \rangle \leq 1\; \text{ for all
  } c\in C \right\} \\&=& {\bf X}^{-T}[C^{\circ}].
\end{eqnarray*}
When $N<n$, the same computation is valid by viewing ${\bf
  X}^{-T}[C^{\circ}]$ as the pre-image of $C^{\circ}$ under ${\bf
  X}^T$.  When $N>n$, ${\bf X}^T$ is injective and ${\bf X}^{-T}$ is
also the inverse of ${\bf X}^T$ on $\mathop{\rm Im}({\bf
  X}^T)=\mathop{\rm ker}({\bf X})^{\perp}$, in which case
\begin{equation}
  \label{eqn:injective}
{\bf X}^{-T}[C^{\circ}]= {\bf X}^{-T}[C^{\circ}\cap \mathop{\rm
    Im}({\bf X}^T)].
\end{equation}
\end{proof}

\begin{remark}When $N<n$, we note that $({\bf X}C)^{\circ}$ denotes polarity in 
$\mathbb{R}^n$ and $({\bf X}C)^{\circ}$ may be unbounded.  When $N\geq
  n$, \eqref{eqn:injective} implies that
\begin{equation}
  \label{eqn:det_image}
  \abs{({\bf X}C)^{\circ}}=\mathop{\rm det}({\bf X}{\bf X}^T)^{-1/2}\abs{C^{\circ}\cap \mathop{\rm Im}({\bf X}^T)}.
\end{equation}
\end{remark}

\subsection{Random slices of finite-dimensional $\ell_p$-balls}

For $p\geq 1$, the centroid body $Z_p(f)$ can be viewed in terms of
limits of images of finite-dimensional $\ell_q$ balls, where
$1/p+1/q=1$.  To fix the notation, for $p\not = 0$, we denote by
$B_p^N$, the $\ell_p$ ball in $\real^N$, i.e.,
\begin{equation}
B_p^N = \left\{ x \in \mathbb{R}^N: \left(\sum\nolimits_{i=1}^N\absinprod{x}{e_i}^p
\right)^{1/p}\leq 1 \right\},
\end{equation}
where $e_1,\ldots,e_N$ is the standard unit vector basis for
$\mathbb{R}^N$. For $p=0$, we set
\begin{equation}
  B_{0}^N=\left\{ x \in \mathbb{R}^N: \left(\prod\nolimits_{i=1}^N
  \absinprod{x}{e_i}\right)^{1/N}\leq 1 \right\}.
\end{equation}
Note that $B_p^N$ is a convex body when $p\in [1,\infty)$ and a
  star-body when $p>0$. When $p\leq 0$, $B_p^N$ is unbounded but
  remains star-shaped. 

Let ${\bf X}$ be an $n\times N$ random matrix with independent column
vectors $X_1,\ldots,X_N$ drawn from $f\in \mathcal{P}_n$.  For $1 \leq
p <\infty$, the empirical $L_p$-centroid body $\mathcal{Z}_{p,N} (f)$
defined above in \eqref{eqn:Zp1} has the equivalent description
\begin{equation*}
  \mathcal{Z}_{p,N}(f)=N^{-1/p} {\bf X} B_q^N,
\end{equation*}
where $1/p+1/q=1$. Indeed,
\begin{eqnarray*}
  \label{eqn:h2}
  h({\bf X} B_q^N,u)=h(B_q^N,{\bf X}^T u) =
  \left(\sum_{i=1}^N\absinprod{X_i}{u}^p\right)^{1/p}.
\end{eqnarray*} 
Using Lemma \ref{lemma:dual} and $1/p+1/q=1$, we have
\begin{equation}
  \label{eqn:Zcirc}
  \mathcal{Z}^{\circ}_{p,N}(f)=N^{1/p} {\bf X}^{-T}[B_p^N],
\end{equation} 
where, as above, ${\bf X}^{-T}[A]$ denotes the pre-image of $A$ under
${\bf X}^T$.  We will mimic identity \eqref{eqn:Zcirc} to realize the
bodies $\mathcal{Z}_{p,N}^{\diamondsuit}(f)$ defined in
\eqref{eqn:rhoZpn} as sections of $B_p^N$ for $p\in(0,1)$.

\begin{lemma}
Let ${\bf X}$ be an $n\times N$ random matrix with independent columns
distributed according to $f\in \mathcal{P}_n$. Then for $p\in(0,1)$,
  \begin{equation*}
    \mathcal{Z}^{\diamondsuit}_{p,N}(f)=N^{1/p} {\bf X}^{-T}[B_p^N].
  \end{equation*} 
\end{lemma}

\begin{proof} 
For $u\in S^{n-1}$, we have
\begin{eqnarray*}
  \rho(\mathcal{Z}_{p,N}^{\diamondsuit}(f),u)  = 
  \rho(N^{1/p}B_p^N,{\bf
    X}^T u)= \rho(N^{1/p} {\bf X}^{-T}[B_p^N],u).
\end{eqnarray*}
\end{proof}


We can similarly view the bodies
$\mathcal{Z}^{\diamondsuit}_{p,\mathcal{C}}(\mathcal{F})$
(cf. \eqref{eqn:rhoZpCp}) using random linear operators.
For $\mathcal{C}=(C_1,\ldots,C_N)$ with $C_i\in\mathcal{K}_s^{m_i}$,
we place them in orthogonal subspaces $\mathbb{R}^{m_i}=\mathop{\rm
  span}\{e_{ij}\}_{j=1}^{m_i}$, $i=1,\ldots,N$. Then for $p \not = 0$,
we define
\begin{equation}
  \label{eqn:calBp}
  \mathcal{B}_p^N(\mathcal{C}) = \left\{(x_1,\ldots,x_N)\in
  \bigoplus_{i=1}^N\mathbb{R}^{m_i}:\left(\sum\nolimits_{i=1}^N
  h^p(C_i,x_i)\right)^{1/p}\leq 1 \right\};
\end{equation}
when $p=\infty$, we replace the sum by $\max_i h(C_i,x_i)$.  For
$p=0$, we set
\begin{equation}
  \label{eqn:calB0}
  \mathcal{B}_0^N(\mathcal{C}) = \left\{(x_1,\ldots,x_N)\in
\bigoplus_{i=1}^N\mathbb{R}^{m_i}:\left(\prod\nolimits_{i=1}^N
  h(C_i,x_i)\right)^{1/N}\leq 1 \right\}.
\end{equation}
When the $C_i's$ are identical copies of $[-1,1]$, we have
$B_p^N=\mathcal{B}_p^N(([-e_1,e_1],\ldots,[-e_N,e_N]))$.  As for
$B_p^N$, the set $\mathcal{B}_p^N(\mathcal{C})$ is a convex body,
star-body or unbounded star-shaped set, according to whether $p\geq
1$, $p\in (0,1)$ or $p\leq 0$, respectively. Note that we have defined
$\mathcal{B}_p^N(\mathcal{C})$ using support functions
$h({C_i},\cdot)$ rather than norms associated to the $C_i$'s, as some
computations are more convenient with this convention. By standard
duality arguments, for $1\leq p,q\leq \infty$ with $1/p+1/q=1$, we
have for $\mathcal{C}=(C_1,\ldots,C_N)$,
\begin{equation}
  \label{eqn:BpNpolar}
  \left(\mathcal{B}_p^N(\mathcal{C})\right)^{\circ} =
  \mathcal{B}_q^N(\mathcal{C}^{\circ}),
\end{equation}
where we have set
$\mathcal{C}^{\circ}=(C_1^{\circ},\ldots,C_N^{\circ})$. We will use
the particular case of $p=1$ and $q=\infty$, combined with
Lemma \ref{lemma:dual} in the following form.

\begin{lemma}
  \label{lemma:XC_polar}
Let $\mathcal{C}=(C_1,\ldots,C_N)$, where $C_i\in
\mathcal{K}_s^{m_i}$, $m_i\geq 1$, and $\mathcal{C}^{\circ}=
(C_1^{\circ},\ldots,C_N^{\circ})$.  Set $M=m_1+\ldots+m_N$.  Let ${\bf
  X} = [{\bf X}_1 \ldots {\bf X}_N]$ be an $n\times M$ matrix with
$n\times m_i$ blocks ${\bf X}_i$ of full rank.  Then
  \begin{equation*}
    \bigcap_{i=1}^N ({\bf X}_iC_i)^{\circ} = ({\bf
      X}\mathcal{B}_1^N(\mathcal{C}^{\circ}))^{\circ}.
  \end{equation*}
\end{lemma}

\begin{proof}By Lemma \ref{eqn:dual},
  \begin{eqnarray*}
    \bigcap_{i=1}^N ({\bf X}_i C_i)^{\circ} = \bigcap_{i=1}^N {\bf
      X}_i^{-T}[C_i^{\circ}] = \bigcap_{i=1}^N\left\{u\in
    \mathbb{R}^n:{\bf X}_i^Tu\in C_i^{\circ}\right\},
  \end{eqnarray*}
  while
  \begin{eqnarray*}
    ({\bf X}\mathcal{B}_1^N(\mathcal{C}^{\circ}))^{\circ} = {\bf
      X}^{-T}[\mathcal{B}_{\infty}^N(\mathcal{C})] =
    \left\{u\in \mathbb{R}^n:\max_{i\leq N}h(C_i,{\bf X}_i^Tu)\leq 1\right\}.
  \end{eqnarray*}
\end{proof}

Using the above notation, the empirical bodies
$\mathcal{Z}_{p,\mathcal{C}}^{\diamondsuit}(\mathcal{F})$ defined in
\eqref{eqn:rhoZpCp} and \eqref{eqn:rhoZpC0} can be realized as
sections of $\mathcal{B}_p^N(\mathcal{C})$ as follows.

\begin{lemma} 
  \label{lemma:emp_radial}
  For $i\in [N]$, let $C_i\in \mathcal{K}_s^{m_i}$, $m_i\geq 1$ and
  let $M=m_1+\ldots+m_N$.  Let ${\bf X} = [{\bf X}_1\cdots {\bf X}_N]$
  be an $n\times M$ random matrix with $n\times m_i$ blocks ${\bf
    X}_i=[X_{i1}\ldots X_{im_i}]$ having independent columns
  ${X_{ij}}$ distributed according to $f_{ij}\in \mathcal{P}_n$.  Then
  for $p\not=0$,
\begin{equation}
  \label{eqn:rad_idp}
  \mathcal{Z}^{\diamondsuit}_{p,\mathcal{C}}(\mathcal{F})=N^{1/p} {\bf X}^{-T}
          [\mathcal{B}_p^N(\mathcal{C}) ],
\end{equation}
and, for $p=0$, 
\begin{equation}
  \label{eqn:rad_id0}
\mathcal{Z}^{\diamondsuit}_{0,\mathcal{C}}(\mathcal{F}) = {\bf
  X}^{-T}[\mathcal{B}_0^N(\mathcal{C})].
\end{equation}
\end{lemma}

\begin{proof}
For $u\in S^{n-1}$,
\begin{equation*}
{\bf X}^Tu=({\bf X}_1^Tu,\ldots,{\bf X}_N^Tu) = ((\langle
X_{1j},u\rangle)_{j=1}^{m_1},\ldots,(\langle X_{Nj},u
  \rangle)_{j=1}^{m_N}).
\end{equation*}
For any set $\mathcal{S}$ in
$\mathbb{R}^{nM}$, we have
\begin{equation*}
  {\bf X}^{-T}[\mathcal{S}] =\left\{u\in \mathbb{R}^n:{\bf X}^Tu\in
  \mathcal{S}\right\}.
\end{equation*}
For $p\not = 0$, we have
  \begin{eqnarray*} \rho(N^{1/p}{\bf X}^{-T}[\mathcal{B}_p^N(\mathcal{C})],u)
    &=&\rho(N^{1/p}\mathcal{B}_{p}^N(\mathcal{C}),{\bf
      X}^Tu)\\&=&\left(\frac{1}{N}\sum_{i=1}^N h^p(C_i,{\bf X}_i^T
    u)\right)^{-1/p}\\
    & = & \rho(\mathcal{Z}_{p,\mathcal{C}}^{\diamondsuit}(\mathcal{F}),u).
  \end{eqnarray*}
  For $p = 0$, we have
\begin{equation*}
  \rho({\bf X}^{-T}[\mathcal{B}_0^N(\mathcal{C})],u)=\prod_{i=1}^N
  h(C_i,{\bf X}_i^Tu)^{-1/N}
  =\rho(\mathcal{Z}_{0,\mathcal{C}}^{\diamondsuit}(\mathcal{F}),u).
\end{equation*}
\end{proof}

\begin{remark} 
  \label{remark:Zpcirc}
  For $p\geq 1$, we have by Lemma \ref{lemma:dual} and
\eqref{eqn:BpNpolar},
\begin{equation}
  \mathcal{Z}^{\diamondsuit}_{p,\mathcal{C}}(\mathcal{F})=
  \mathcal{Z}^{\circ}_{p,\mathcal{C}}(\mathcal{F})= N^{1/p} {\bf
    X}^{-T} [\mathcal{B}_p^N(\mathcal{C}) ] = N^{1/p}({\bf
    X}\mathcal{B}_q^N(\mathcal{C}^{\circ}))^{\circ}.
\end{equation}
\end{remark}

\begin{remark}
  \label{rem:infinite}
For $p\leq 0$, the bodies
$\mathcal{Z}_{p,\mathcal{C}}^{\diamondsuit}(\mathcal{F})$ are
pre-images of slices of unbounded sets and hence need not be bounded.
This is reflected in our notation as their radial functions take the
value $+\infty$.  When $m_j=\mathop{\rm dim}(C_j)<n$, the matrix ${\bf
  X}_j^T$ has a non-trivial kernel and, for $p\not=0$, 
\begin{equation*}
  \rho({\bf
    X}^{-T}[\mathcal{B}_p^N(\mathcal{C})],u)=\left(\sum_{i=1}^N
  h^{-\abs{p}}(C_i,{\bf X}_i^Tu)\right)^{1/\abs{p}}\geq h^{-1}(C_j,{\bf X}_j^Tu),
\end{equation*}
which is infinite for $u\in\mathop{\rm ker}({\bf X}_j^T)$ and
arbitrarily large in any neighboorhood of such $u$. When each $C_i$
has dimension $m_i\geq n$, the absolute continuity of $f_{ij}$ ensures
that the $n\times m_i$ matrix ${\bf X_i}$ has rank $n$ a.s.. This
implies that $h(C,{\bf X}_i^T\cdot)>0$ a.s., hence each summand in the
radial function
$\rho(\mathcal{Z}_{p,\mathcal{C}}^{\diamondsuit}(\mathcal{F}),\cdot)$
is necessarily finite a.s.. 
\end{remark}

For $N$-tuples of origin-symmetric convex sets
$\mathcal{C}=(C_1,\ldots,C_N)$ and $\mathcal{D}=(D_1,\ldots,D_N)$,
with $\mathop{\rm dim}(C_i)\leq \mathop{\rm dim}(D_i)$, we will write
\begin{equation*}
  \mathcal{C} \subseteq \mathcal{D} \Longleftrightarrow C_i\subseteq
  D_i \text{ for all } i=1,\ldots,N.
\end{equation*}

\begin{lemma}
  \label{lemma:emp_ab}
  Let $\mathcal{F}=(f_{ij})\subseteq \mathcal{P}_n$ and $-1\leq p,
  p_1, p_2\leq \infty$.  Let $\mathcal{C}=(C_1,\ldots,C_N)$ with
  $C_i\in \mathcal{K}_s^{m_i}$, $m_i\geq 1$,  for $i\in [N]$.
  \begin{itemize}
    \item[(a)] If $p_1\leq p_2$, then
      \begin{equation}
        \label{eqn:Zp1p2}
  \mathcal{Z}_{p_2,\mathcal{C}}^{\diamondsuit}(\mathcal{F})\subseteq
  \mathcal{Z}_{p_1,\mathcal{C}}^{\diamondsuit}(\mathcal{F}).
\end{equation}
\item[(b)] If $\mathcal{D}=(D_1,\ldots,D_N)$ with $D_i\in
  \mathcal{K}_s^{m_i'}$, $m_i'\geq m_i$, for $i\in [N]$, and
  $\mathcal{C}\subseteq \mathcal{D}$, then
  \begin{equation}
    \label{eqn:ZCD}
  \mathcal{Z}_{p,\mathcal{C}}^{\diamondsuit}(\mathcal{F})\supseteq
  \mathcal{Z}_{p,\mathcal{D}}^{\diamondsuit}(\mathcal{F}).
\end{equation}
\end{itemize}
  \end{lemma}

\begin{proof}
  Part $(a)$ is a consequence of H\"{o}lder's inequality, which gives
  monotonicity of the normalized means in the definition of
  $\rho(\mathcal{Z}_{p,\mathcal{C}}^{\diamondsuit}(\mathcal{F}),u)$
  (cf. \eqref{eqn:rhoZpCp} and   \eqref{eqn:rhoZpC0}). 

For part $(b)$, $C_i\subseteq D_i$ is equivalent to
$h(C_i,\cdot)\leq h(D_i,\cdot)$ for each $i$, which implies
\eqref{eqn:ZCD}.
\end{proof}

\subsection{Convergence of volumes}

The next proposition provides sufficient conditions to obtain the
volume of $Z_{p,C}^{\diamondsuit}(f)$ as a limit of the expected
volumes of the random bodies
$\mathcal{Z}_{p,\mathcal{C}}^{\diamondsuit}(\mathcal{F})$.

\begin{prop}
\label{prop:ZpCUI}
For $i\in \mathbb{N}$, let $C_i\in \mathcal{K}_s^{m_i}$, $m_i\geq 1$
and $(f_{ij})\subseteq \mathcal{P}_n$, $j\in[m_i]$.  For $N\in
\mathbb{N}$, let $\mathcal{C}_N=(C_1,\ldots,C_N)$ and
$\mathcal{F}_N=((f_{ij})_{j=1}^{m_i})_{i=1}^N$.  Assume that
\begin{itemize}
\item[a.]there is an $r_0>0$ such that $r_0B_2^{m_i}\subseteq C_i$ for
  each $i$.
  \item[b.] $f_{ij}$ are supported on a common compact set and 
    $\sup_{i,j}\norm{f_{ij}}_{\infty}<\infty$.
\end{itemize}
If $p\in [0,1]$, or
 $p\in[-1,0)$ and $m_i\geq n+1$ for each $i$, then for any
   $\varepsilon \in (0,1)$,
\begin{equation}
  \label{eqn:1eps}
  \sup_{N\geq n+1}\sup_{u\in S^{n-1}}\mathbb{E}
  \rho^{n+\varepsilon}(\mathcal{Z}_{p,\mathcal{C}_N}^{\diamondsuit}(\mathcal{F}_N),u)
  <\infty,
\end{equation} and hence
\begin{equation}
  \label{eqn:vol_finite}
 \sup_{N\geq n+1}
 \mathbb{E}\abs{\mathcal{Z}_{p,\mathcal{C}_N}^{\diamondsuit}(\mathcal{F}_N)}<\infty.
\end{equation}
Furthermore, if $C_1,C_2,\ldots$ are
  copies of a given convex body $C$ of dimension $m$ and
  $f_{ij}$ are identical and satisfy \eqref{eqn:1eps},
  then
\begin{equation}
  \label{eqn:Zplim}
  \abs{Z_{p,C}^{\diamondsuit}(f)}=\lim_{N\rightarrow\infty}\mathbb{E}
  \abs{\mathcal{Z}_{p,\mathcal{C}_N}^{\diamondsuit}({\mathcal{F}_N})}.
\end{equation}
\end{prop}

\begin{proof} 
Without loss of generality, we may assume that $r_0=1$.  By assumption
(b), we can fix a Gaussian density $\phi_{\alpha}$ and a constant
$A>0$ such that for each $i,j$,
\begin{equation}
  \label{eqn:Gauss_bound}
\frac{1}{A}f_{ij}(x)\leq
\phi_{\alpha}(x)=\frac{1}{(2\pi\alpha^2)^{n/2}}e^{-\norm{x}_2^2/2\alpha^2} \quad (x\in \mathbb{R}^n).
\end{equation}  

Fix $\varepsilon >0$ and $u\in S^{n-1}$.  Assume first that
$p\in[0,1]$. By Lemma \ref{lemma:emp_ab}, we need only treat the case
$p=0$, $m_i=1$ for $i=1,\ldots,N$, and
$\mathcal{C}_N=([-e_i,e_i])_{i=1}^N$.  In the notation of Lemma
\ref{lemma:emp_radial}, this means that
$\mathcal{F}_N=(f_{i1})_{i=1}^N$ and ${\bf X}_i=[X_{i1}]$ are $n\times
1$ matrices.  By Fubini's theorem,
\begin{eqnarray*}
\mathbb{E}\rho^{n+\varepsilon} \big(
\mathcal{Z}_{0,\mathcal{C}_N}^\diamondsuit (\mathcal{F}_N),u \big) & =
&\prod_{i=1}^N \mathbb{E} \absinprod{X_{i1}}{u}^{-(n+\varepsilon)/N}.
\end{eqnarray*}
Set $\tau=(n+\varepsilon)/N$.  Let $g_1,\ldots,g_N$ be i.i.d. standard
Gaussian vectors in $\mathbb{R}^n$. Fix $i\in \{1,\ldots,N\}$. Then
$\langle g_i, u\rangle$ is a standard Gaussian random variable.
Assume first that $N\geq 2(n+\varepsilon)$ so that $\tau\leq
1/2$. By H\"{o}lder's inequality,
\begin{eqnarray*}
\mathbb{E}_{X_{i1}}\absinprod{X_{i1}}{u}^{-\tau}\leq
\bigl(\mathbb{E}\absinprod{X_{i1}}{u}^{-1/2}\bigr)^{2\tau}.
\end{eqnarray*} Using \eqref{eqn:Gauss_bound} and 
the notation for $b_{n,s}$ from \eqref{eqn:bns}, we have
\begin{equation*}
  A^{-1}\alpha^{1/2}\mathbb{E}_{X_{i1}}\absinprod{X_{i1}}{u}^{-1/2}
  \leq \mathbb{E}_{g_i}\absinprod{g_i}{u}^{-1/2}=b_{1,1/2},
\end{equation*}
hence for $N\geq 2(n+\varepsilon)$, 
\begin{eqnarray}
  \label{eqn:N>2n}
\mathbb{E}\rho^{n+\varepsilon} \big(\mathcal{Z}_{0,\mathcal{C}_N}^\diamondsuit
(\mathcal{F}_N), u \big)\leq (A\alpha^{-1/2}b_{1,1/2})^{2\tau N}=
(A\alpha^{-1/2}b_{1,1/2})^{2(n+\varepsilon)}.
\end{eqnarray}
Assume now that $n+1\leq N< 2(n+\varepsilon)$. Then $\tau<1$ and
$b_{1,\tau}<\infty$. By \eqref{eqn:Gauss_bound},
\begin{equation*}
  A^{-1}\alpha^{\tau}\mathbb{E}_{X_{i1}}\absinprod{X_{i1}}{u}^{-\tau}
  \leq \mathbb{E}_{g_i}\absinprod{g_i}{u}^{-\tau}=b_{1,\tau}.
\end{equation*}
It follows that for $n+1\leq N< 2(n+\varepsilon)$,
\begin{equation}
  \label{eqn:N<2n}
\mathbb{E}\rho^{n+\varepsilon} \big(
\mathcal{Z}_{0,\mathcal{C}_N}^\diamondsuit (\mathcal{F}_N), u \big)\leq (A
\alpha^{-\tau}b_{1,\tau})^{N}\leq \max(1,(A
\alpha^{-\tau}b_{1,\tau})^{2(n+\varepsilon)}).
\end{equation}
The bounds in \eqref{eqn:N>2n} and \eqref{eqn:N<2n} are 
independent of $u$ and $N$. This establishes \eqref{eqn:1eps} for
$p\in[0,1]$.

Assume now that $p\in[-1,0)$. By Lemma \ref{lemma:emp_ab}, we can
  assume that $p=-1$, $m_i=n+1$, for $i=1,\ldots,N$ and
  $\mathcal{C}=(B_2^{n+1})_{i=1}^N$.  Set $s=n+\varepsilon$ and let
  $s'$ be defined by $1/s+1/s'=1$. By H\"{o}lder's inequality,
\begin{eqnarray*}
    \rho^{n+\varepsilon}(\mathcal{Z}_{-1,\mathcal{C}_N}^{\diamondsuit}(\mathcal{F}_N),u)
     = \left(\frac{1}{N}\sum_{i=1}^N \norm{{\bf
      X}_i^Tu}_2^{-1}\right)^{n+\varepsilon}
     \leq  N^{s/s'-s} \sum_{i=1}^N \norm{{\bf X}_i^Tu}_2^{-(n+\varepsilon)}.
\end{eqnarray*}
For $i\in [N]$, let ${\bf G}_i$ be i.i.d. $n\times (n+1)$ random
matrices with i.i.d. $N(0,1)$ entries.  For $i\in [N]$,
\begin{eqnarray*}
  A^{-(n+1)}\alpha^{n+\varepsilon}\mathbb{E}_{{\bf X_i}}\norm{{\bf X}_i^Tu}^{-(n+\varepsilon)}_2 \leq
  \mathbb{E}_{{\bf G_i}}\norm{{\bf G}_i^Tu}^{-(n+\varepsilon)}_2=  b_{n+1,n+\varepsilon}.
\end{eqnarray*}
Using $N^{1-s}N^{s/s'}=1$, \eqref{eqn:1eps} now follows from
\begin{eqnarray*}
  \mathbb{E}_{{\bf X}}
  \rho^{n+\varepsilon}(\mathcal{Z}_{-1,\mathcal{C}_N}
  ^{\diamondsuit}(\mathcal{F}_N), u)&\leq &
  A^{n+1}\alpha^{-(n+\varepsilon)} b_{n+1, n+\varepsilon};
\end{eqnarray*}
here we have used that $m_i=\mathop{\rm dim}(C_i)\geq n+1$, which ensures
finiteness of $b_{n+1,n+\varepsilon}$. To justify
\eqref{eqn:vol_finite}, for general $\mathcal{C}_N$ and
$\mathcal{F}_N$, set $\delta = \varepsilon/n$ so that
$n(1+\delta)=n+\varepsilon$. By H\"{o}lder's inequality,
\begin{eqnarray*}
  \left(\int_{S^{n-1}}\mathbb{E}
  \rho^n(\mathcal{Z}_{p,\mathcal{C}_N}^{\diamondsuit}(\mathcal{F}_N),u)du\right)^{1+\delta}&
  \leq & \int_{S^{n-1}}\left(\mathbb{E}
  \rho^n(\mathcal{Z}_{p,\mathcal{C}_N}^{\diamondsuit}(\mathcal{F}_N),u)\right)^{1+\delta}du\\ &
  \leq &
  \int_{S^{n-1}}\mathbb{E}\rho^{n+\varepsilon}(\mathcal{Z}_{p,\mathcal{C}_N}^{\diamondsuit}(\mathcal{F}_N),u)du.
\end{eqnarray*}
Therefore, \eqref{eqn:vol_finite} follows from
\begin{equation}
  \label{eqn:UI_int}
  \left(\mathbb{E}\abs{\mathcal{Z}_{p,\mathcal{C}_N}^{\diamondsuit}(\mathcal{F}_N)}\right)^{1+\delta} \leq \omega_n^{1+\delta}
  \sup_{u\in S^{n-1}}\mathbb{E}\rho^{n+\varepsilon}(\mathcal{Z}_{p,\mathcal{C}_N}^{\diamondsuit}(\mathcal{F}_N),u).
  \end{equation}

Towards proving \eqref{eqn:Zplim}, we fix $u\in S^{n-1}$, identical
bodies $C_i=C$ of dimension $m$ and $f_{ij}=f$. For $p\not = 0$, the family of
i.i.d. random variables $\left\{h^p(C,{\bf X}_i^Tu)\right\}_{i\in
  \mathbb{N}}$ has finite first moment, i.e.,
\begin{equation}
  \label{eqn:Ehp}
  \mathbb{E}h^p(C,{\bf X}_i^Tu) =
  \int_{(\mathbb{R}^n)^m}h^{p}(C,(\langle x_i, u
  \rangle)_{i=1}^m)\prod_{i=1}^m f(x_i)d\overline{x}<\infty.
\end{equation}
  Indeed, for $p>0$, this is a direct consequence of $f$ being bounded
  and compactly supported. For $p<0$, the function
  $\mathbb{E}h^p(C,{\bf
    X}_i^T\cdot)=\rho^{-p}(Z_{p,C}^{\diamondsuit}(f),\cdot)$ is
  integrable by part (c) of Lemma \ref{lemma:Zbasic}; in particular,
  \eqref{eqn:Ehp} holds for all $u$ outside of a null set on $S^{n-1}$
  (henceforth disregarded).  Thus by Proposition \ref{prop:SLLN}, for our fixed $u\in S^{n-1}$, 
\begin{equation*}
  \frac{1}{N}\sum_{i=1}^N h^p(C,{\bf X}_i^Tu) \rightarrow
  \mathbb{E}h^p(C,{\bf X}_i^Tu) =
  \rho^{-p}(Z_{p,C}^{\diamondsuit}(f),u) \quad \text{(a.s.)};
\end{equation*}
similarly, for $p=0$, the i.i.d. collection $\left\{\log h(C,{\bf
  X}_i^Tu)\right\}_{i\in \mathbb{N}}$ satisfies
\begin{equation*}
\mathbb{E}\abs{\log h(C,{\bf X}_i^Tu)}=\int_{(\mathbb{R}^n)^m} \abs{\log
  h(C,(\langle x_i,u\rangle))}\prod_{i=1}^m
f(x_i)d\overline{x}<\infty,
\end{equation*}hence
\begin{equation*}
  \frac{1}{N}\sum_{i=1}^N \log h(C,{\bf X}_i^Tu) \rightarrow \mathbb{E}
  \log h(C,{\bf X}_i^Tu)  \quad \text{(a.s.)}.
\end{equation*}
In all cases, we have
\begin{equation*}
  \rho^n(\mathcal{Z}_{p,\mathcal{C}_N}^{\diamondsuit}(\mathcal{F}_N),u)
  {\rightarrow}
  \rho^n(Z_{p,C}^{\diamondsuit}(f),u)\quad {\text{(a.s.)}}.
\end{equation*}
Using \eqref{eqn:1eps}, the collection
$\left\{\rho^n(\mathcal{Z}_{p,\mathcal{C}_N}^{\diamondsuit}(\mathcal{F}_N),u):
N\geq n+1\right\}$ (for our fixed $u$) is bounded in
$L_{1+\delta}$, where, as above, $\delta = \varepsilon/n$. By
Proposition \ref{prop:UI} and Remark \ref{remark:UI}, as $N\rightarrow
\infty$,
\begin{equation}
  \label{eqn:exprho}
  \mathbb{E}\rho^n(\mathcal{Z}_{p,\mathcal{C}_N}^{\diamondsuit}(\mathcal{F}_N),u)
  \rightarrow \mathbb{E}\rho^n(Z_{p,C}^{\diamondsuit}(f),u) =
\rho^n(Z_{p,C}^{\diamondsuit}(f),u).
\end{equation}

Lastly, the collection $\left\{\mathbb{E}
\rho^n(\mathcal{Z}_{p,\mathcal{C}_N}^{\diamondsuit}(\mathcal{F}_N),\cdot):N\geq
n+1\right\}$ is uniformly integrable on
$(S^{n-1},\sigma)$ (by the inequality preceding
\eqref{eqn:UI_int}). Using \eqref{eqn:exprho}, Proposition
\ref{prop:UI} and Fubini's theorem, we get
\begin{eqnarray*}
\abs{Z_{p,C}^{\diamondsuit}(f)} & = &
\omega_n\int_{S^{n-1}}\rho^n(Z_{p,C}^{\diamondsuit}(f),u)du \\&=&
\omega_n\lim_{N\rightarrow
  \infty}\int_{S^{n-1}}\mathbb{E}\rho^n(\mathcal{Z}_{p,\mathcal{C}_N}^{\diamondsuit}(\mathcal{F}_N),u)du\\
&= &\lim_{N\rightarrow
  \infty}\mathbb{E}\abs{\mathcal{Z}_{p,\mathcal{C}_N}^{\diamondsuit}(\mathcal{F}_N)}, 
\end{eqnarray*}
which establishes \eqref{eqn:Zplim} and completes the proof of the
proposition.
\end{proof}

\subsection{Empirical $L_p$-intersection bodies}

\label{section:emp_int}

In this section, we show how particular choices of $\mathcal{C}$ and
$\mathcal{F}$ in the bodies
$\mathcal{Z}_{p,\mathcal{C}}^{\diamondsuit}(\mathcal{F})$ lead
naturally to empirical versions of $L_p$-intersection bodies.  As
mentioned, unit balls of normed spaces that embed in $L_p$,
$p\in[-1,1]$ can be obtained as limits of $p$-radial sums of
ellipsoids \cite{GooWei95, KKYY07}.  The next proposition can be seen
as a complementary volumetric {\it random} approximation.  Since our
main interest is when $p=-1$, we have stated this only for
$p\in[-1,0)$; similar considerations lead to an analogous result for
  $p>0$.

For $p\in[-1,0)$ and $\alpha>0$, we define the empirical
  $L_p^{\alpha}$-intersection body
  $\mathcal{I}_{\abs{p},N}^{\alpha}(f)$ via
\begin{equation*}
  \rho(\mathcal{I}_{\abs{p},N}^{\alpha}(f), u) = \frac{1}{N} \sum_{i=1}^N
  \rho^{\abs{p}}(\mathcal{E}^{\alpha}(X_i),u),
\end{equation*}
where $X_1,\ldots,X_N$ are i.i.d. with density $f\in\mathcal{P}_n$ and
$\mathcal{E}^{\alpha}(X_i)=([-X_i,X_i]+_2 \alpha B_2^n)^{\circ}$.
  
\begin{prop}
  \label{prop:IZ}
  Let $f$ be a compactly supported function in $\mathcal{P}_n$. Let
  $p\in[-1,0)$ and $\alpha >0$.  Then for $N\geq n+1$,
\begin{equation*}
    \mathbb{E}\abs{\mathcal{I}_{\abs{p},N}^{\alpha}(f)} =
    \lim_{m\rightarrow \infty}
    \mathbb{E}\abs{\mathcal{Z}^{\diamondsuit}_{p,\mathcal{C}_{m}^\alpha}
      (\mathcal{F}_{m})},
\end{equation*}
where $\mathcal{C}_{m}^{\alpha}=(C_m^{\alpha})_{i=1}^N$ and
$\mathcal{F}_{m}=((f_{ij})_{j=1}^{m+1})_{i=1}^N$ are
given by
\begin{equation*}C_m^{\alpha}=[-e_1,e_1]+_2 \alpha\mathop{\rm
      conv}\{\pm e_j\}_{j=2}^{m+1}, \quad \quad  f_{ij} = \begin{cases}
    f & \text{ if } i\in [N], j=1\\
    \omega_n^{-1}\chi_{B_2^n} & \text{ if } i\in[N], j>1.
    \end{cases}
    \end{equation*}
\end{prop}

\begin{proof}
  Let ${\bf X}_1, \ldots {\bf X}_N$ be i.i.d.  $n\times (m+1)$ random
  matrices with ${\bf X}_i=[X_{i1} Z_{i1}\cdots Z_{im}]$, where
  $X_{i1}$ has density $f_{i1}=f$, and $Z_{ij}$ has density
  $f_{i(j+1)}=\omega_n^{-1}\chi_{B_2^n}$, and the columns are
  independent. Then for $i=1,\ldots,N$,
  \begin{equation*}
    {\bf X}_i C_m^{\alpha} =[-X_{i1},X_{i1}]+_2 \alpha \mathop{\rm
      conv}\{\pm Z_{ij}\}_{j=1}^m,
  \end{equation*}
  which for $m\rightarrow \infty$, converges a.s. in the Hausdorff
  metric to $[-X_{i1},X_{i1}]+_2 \alpha B_2^n$. 
  For $u\in S^{n-1}$, we have as $m\rightarrow \infty$,
  \begin{eqnarray*}
    \frac{1}{N}\sum_{i=1}^N h^{-\abs{p}}({\bf X}_i
    C_m^{\alpha},u) \rightarrow \frac{1}{N}\sum_{i=1}^N
    h^{-\abs{p}}([-X_{i1},X_{i1}]+_2 \alpha B_2^n, u)\quad \text{(a.s.)},
  \end{eqnarray*}
and hence
\begin{equation*}
  \rho^n(\mathcal{Z}^{\diamondsuit}_{p,\mathcal{C}_{m}^\alpha}
  (\mathcal{F}_{m}),u)\rightarrow
  \rho^n(\mathcal{I}_{\abs{p},N}^{\alpha}(f),u)\quad\text{(a.s.)}.
\end{equation*}
For $m\geq n$, the latter convergence is dominated by
$\rho^n(\mathcal{Z}^{\diamondsuit}_{-1,\mathcal{C}_{n}^{\alpha}}
(\mathcal{F}_{n}),u)$ (cf. \eqref{eqn:Zp1p2}).  The inradius of
$C_{n}^{\alpha}$ is $\min(1,\alpha/\sqrt{n})$. Using Proposition
\ref{prop:ZpCUI} with fixed $N\geq n+1$,
\begin{equation*}
  \int_{S^{n-1}}\mathbb{E}
  \rho^n(\mathcal{Z}^{\diamondsuit}_{-1,\mathcal{C}_n^{\alpha}}
  (\mathcal{F}_{n}),u)du<\infty.
\end{equation*}
By dominated convergence, we get
\begin{equation*}
  \mathbb{E}\int_{S^{n-1}}
  \rho^n(\mathcal{I}_{\abs{p},N}^{\alpha}(f),u)du
  =\lim_{m\rightarrow \infty}\mathbb{E} \int_{S^{n-1}}
  \rho^n(\mathcal{Z}^{\diamondsuit}_{p,\mathcal{C}_{m}^{\alpha}}
  (\mathcal{F}_{m}),u)du.
\end{equation*}
\end{proof}

\begin{prop}
  \label{prop:Ipa}
Let $f\in \mathcal{P}_n$, $p\in [-1,0)$, and $\alpha>0$. Then
\begin{equation*}
\abs{I_{\abs{p}}^{\alpha}(f)}=\lim_{N\rightarrow
  \infty}\mathbb{E}\abs{\mathcal{I}_{\abs{p},N}^{\alpha}(f)}.
\end{equation*}      
\end{prop}

\begin{proof}
  Fix $u\in S^{n-1}$. Since $f\in \mathcal{P}_n$, the random variables
  $\left(\absinprod{X_i}{u}^2 +\alpha^2 \norm{u}_2^2\right)^{-\abs{p}/2}$
  have finite first moment.  By the law of large numbers, as
  $N\rightarrow \infty$, we have 
  \begin{equation*}
      \frac{1}{N}\sum_{i=1}^N \left(\absinprod{X_i}{u}^2 +\alpha^2
      \norm{u}_2^2\right)^{-1/2} \rightarrow
      \int_{\mathbb{R}^n}\left(\absinprod{x}{u}^2+\alpha^2
      \norm{u}_2^2\right)^{-1/2}f(x)dx\quad \text{(a.s.)},
  \end{equation*}
  hence
  \begin{equation*}
    \rho^n(\mathcal{I}_{\abs{p},N}^{\alpha}(f),u)\rightarrow 
\rho^n(I_{\abs{p}}^{\alpha}(f),u)\quad \text{(a.s.)}.
  \end{equation*}
  Since $\rho(I_{\abs{p},N}^{\alpha}(f),u)\leq 1/\alpha$ for each $u$,
  we can use dominated convergence to get
  \begin{equation*}
    \mathbb{E} \int_{S^{n-1}}
    \rho^n(\mathcal{I}_{\abs{p},N}^{\alpha}(f),u)du\rightarrow
    \mathbb{E}\int_{S^{n-1}}\rho^n(I_{\abs{p}}^{\alpha}(f),u)du
    =\abs{I_{\abs{p}}^{\alpha}(f)}.
  \end{equation*}
\end{proof}

\section{Volume formulas}
\label{section:volume_formulas}
As mentioned, our work is inspired by a formula for the volume of
sections of $B_p^N$, $p\in (0,2)$, due to Nayar and Tkocz \cite{NT20}.
We will recall the basic ingredients and then derive a formula for the
volume of the random sets
$\mathcal{Z}_{p,\mathcal{C}}^{\diamondsuit}(\mathcal{F})$. For $p\leq
0$, we will present an alternative path and complementary volume
formulas.

\subsection{Volume via Gaussian mixtures for $p>0$}

Recall that for $0 < \alpha < 1$, a positive random variable $w$ is
called \emph{normalized positive $\alpha$-stable} if 
\begin{equation}
\eval e^{-tw} = e^{-t^{\alpha}} \quad (t>0).
\end{equation}
We will denote the density of such a random variable by $g_{\alpha}$;
for background on stable random variables, see \cite{Zol86}. The
following Nayar--Tkocz volume formula was proved in \cite{NT20}, where
it is stated explicitly for $p=1$ and explained how the same method
applies to $p\in(0,2)$.
\begin{prop}
  \label{prop:NT_Bp}
  Let $0<p<2$ and let ${X}$ be a $n \times N$ matrix  with
  columns $x_1, \cdots, x_N$ spanning $\mathbb{R}^n$.  Let $W=(w_1,
  \cdots, w_N)$ be a random vector with i.i.d. entries $w_i$ having
  common density proportional to $s\mapsto s^{-1/2} g_{p/2} (s)$.
  Then
  \begin{equation}
    \label{eq:dim-formula}
    \frac{\abs{B_p^N \cap \mathop{\rm Im}({X}^{T})}}
         {\mathop{\rm det}({ X} { X}^T)^\frac{1}{2}} =
    a_{N,n,p}\pi^{n/2} 
    \mathbb{E}_W \sqrt{w_1\cdots w_N} \left( \mathop{\rm
      det}\left(\sum\nolimits_{i=1}^N w_i x_i x_i^T \right)\right)^{-\frac{1}{2}}.
  \end{equation}
  where $a_{N,n,p} = \pi^{-N/2}\Gamma\left(1 + 1/p
  \right)^{N}\Gamma \left(1+n/p\right)^{-1}$.
\end{prop}

The proof of the formula relies on two ingredients. The first is that
the volume of a star-body $K$ in $\mathbb{R}^n$ with radial function
$\rho(K,\cdot)$ is given by
\begin{equation}
  \label{eqn:vol_exp}
  \abs{K}=c_{n,p}\int_{\mathbb{R}^n}\exp\left(-\rho^{-p}(K,x)\right)dx,
\end{equation}where $c_{n,p}=\Gamma(1+n/p)^{-1}$.
The second ingredient is the following fact from \cite[Lemma
  23]{ENT18}: if $\xi$ is a standard Gaussian random variable,
independent of a positive random variable $w$ with density
proportional to $t\mapsto t^{-1/2} g_{p/2} (t)$, then
$\frac{1}{\sqrt{2w}} \xi$ has density $\left[2 \Gamma
  \left(1+1/p\right)\right]^{-1} e^{-|t|^p}$ and
\begin{equation}
  \label{eqn:exp_xp}
e^{-|x|^p} = d_p
\mathbb{E}_w \sqrt{w} e^{-w x^2 } \quad (x\in \mathbb{R}),
\end{equation}
where $d_p=\Gamma \left(1 + 1/p \right)/ \sqrt{\pi}$ (as can be seen
by integrating \eqref{eqn:exp_xp} on $\mathbb{R}$).

We will adapt the Nayar-Tkocz argument to derive a volume formula for
$\mathcal{Z}_{p,\mathcal{C}}^{\diamondsuit}(\mathcal{F})$ for $p\in
(0,2)$, using the pre-image interpretation in \eqref{eqn:rad_idp}.

\begin{prop}
\label{prop:NT_ext}
Let $\mathcal{C}=(C_1,\ldots,C_N)$, $\mathcal{F}$ and ${\bf X}$ be as
in Lemma \ref{lemma:emp_radial}.  Let $0<p<2$ and let $W=(w_1, \cdots,
w_N)$ be a random vector with i.i.d. entries $w_i$ having a common
density proportional to $s\mapsto s^{-1/2} g_{p/2} (s)$.  Set
$\mathcal{C}_W^{\circ}=((\sqrt{w_1}C_1)^{\circ},\ldots,(\sqrt{w_N}C_N)^{\circ})$. Then
  \begin{equation*}
    \abs{\mathcal{Z}_{p,\mathcal{C}}^{\diamondsuit}(\mathcal{F})}=
    a_{N,n,p}c_{n,2}^{-1} N^{n/p}\mathbb{E}_{W}{\sqrt{w_1\cdots
        w_N}}\abs{\bigl({\bf
        X}\mathcal{B}_2^N(\mathcal{C}_W^{\circ})\bigr)^{\circ}}.
  \end{equation*}
\end{prop}

\begin{proof}
  Note that
  $(\mathcal{B}_2^N(\mathcal{C}_W))^{\circ}=\mathcal{B}_2^N(\mathcal{C}_W^{\circ})$
  (cf. \eqref{eqn:calBp}), hence
  \begin{eqnarray}
    \label{eqn:B2}
    \sum_{i=1}^N
    h^2(\sqrt{w_i}C_i,{\bf X}_i^Tu)
    =h^2({\bf X}\mathcal{B}_2^N(\mathcal{C}_W^{\circ}), u)
  \end{eqnarray}
  Using \eqref{eqn:vol_exp}--\eqref{eqn:B2}, we have
  \begin{eqnarray*}
    c_{n,p}^{-1}\abs{{\bf X}^{-T}[\mathcal{B}_p^N(\mathcal{C})]} &=&
    \int_{\mathbb{R}^n}\exp\left(-\rho^{-p}(\mathcal{B}_p^N(\mathcal{C}),{\bf
      X}^T
    u)\right)du\\ &=&\int_{\mathbb{R}^n}\prod_{i=1}^N\exp\left(-
    h^p(C_i,{\bf X}^T_i u) \right)du\\ &=& d_p^N\int_{\real^n}
    \eval_{W} \prod_{i=1}^N {\sqrt{w_i}}\exp\left(-w_ih^2(C_i, {\bf
      X}_i^Tu)\right) du \\ &=& d_p^N\int_{\real^n} \eval_{W}
         {\sqrt{w_1\cdots w_N}}\exp\left(-\sum\nolimits_{i=1}^N
         h^2(\sqrt{w_i}\;C_i,{\bf X}_i^Tu)\right)du\\
         & = & c_{n,2}^{-1}d_p^N \eval_{W} {\sqrt{w_1\cdots w_N}}\abs{{\bf
              X}^{-T}[\mathcal{B}_2^N(\mathcal{C}_W)]},
  \end{eqnarray*}
where $\mathcal{C}_W=(\sqrt{w_1}C_1,\ldots,\sqrt{w_N}C_N)$ and we used
\eqref{eqn:vol_exp} again in the last equality.  The result now
follows from Lemmas \ref{lemma:dual} and \ref{lemma:emp_radial} and
the identity $a_{N,n,p}=c_{n,p}d_p^N$.
\end{proof}

\begin{remark}To see that the latter proposition implies 
\eqref{eq:dim-formula}, we take $C_i=[-e_i,e_i]$ and write ${\bf
  X}_W=[\sqrt{w_1}X_1,\ldots,\sqrt{w}_N X_N]$ so that ${\bf X}_W B_2^N =
{\bf X}\mathcal{B}_2^N(\mathcal{C}_W^{\circ})$. By \eqref{eqn:det_image}, 
\begin{equation*}
  \abs{({\bf X}_WB_2^N)^{\circ}} = \omega_n \left( \mathop{\rm
    det}\left(\sum\nolimits_{i=1}^N w_i X_i X_i^T \right)\right)^{-
    \frac{1}{2}}.
\end{equation*}
When $p=1$ in \eqref{eq:dim-formula}, $w_i$ is the reciprocal of an
exponential random variable \cite{NT20} and we have maintained this
convention here, though the exact normalization is immaterial in what follows.
\end{remark}

\subsection{Volume via Gaussian measure for $p=0$}

The set $\mathcal{Z}_{0,\mathcal{C}}^{\diamondsuit}(\mathcal{F})$ can
be treated as a limiting case of
$\mathcal{Z}_{p,\mathcal{C}}^{\diamondsuit}(\mathcal{F})$ when
$p\rightarrow 0$ but it will be handy to derive a different volume
formula using the pre-image representation \eqref{eqn:rad_id0}
directly. This approach will also be helpful for $p<0$.  The formula
involves standard Gaussian measure $\gamma_n$ and negative moments of
the Gaussian random vectors $b_{n,s}$ defined in \eqref{eqn:bns}.

\begin{prop}
\label{prop:rhoZ0C}
Let $\mathcal{C}=(C_1,\ldots,C_N)$, $\mathcal{F}$ and ${\bf X}$ be as
in Lemma \ref{lemma:emp_radial}. For $t=(t_1,\ldots,t_N)$ in
$\mathbb{R}_{+}^N$ and $s>0$, set
$\mathcal{C}_{s,t}^{\circ}=\left((t_1^{N/s}
C_1)^{\circ},\ldots,(t_N^{N/s}
C_N)^{\circ}\right)$.  Then
\begin{equation}
  \label{eqn:Z0C_vol}
  \mathbb{E}\abs{\mathcal{Z}_{0,\mathcal{C}}^{\diamondsuit}(\mathcal{F})}=
  \lim_{s\rightarrow
    n^{-}}b^{-1}_{n,s}\int_{\mathbb{R}_{+}^N}\mathbb{E}_{{\bf
      X}}\gamma_n\left(({\bf
    X}\mathcal{B}_1^N(\mathcal{C}_{s,t}^{\circ}))^{\circ}\right)dt.
\end{equation}
\end{prop}

\begin{proof}
We will first show that for $u\in \mathbb{R}^n\backslash\{0\}$, 
\begin{equation}
\label{eqn:rhoZ0C}
\rho^s(\mathcal{Z}_{0,\mathcal{C}}^{\diamondsuit}(\mathcal{F}), u)
= \int_{\mathbb{R}_{+}^N} \left[ u \in 
  	\left( \mathbf{X} \mathcal{B}_1^N \left(
  	\mathcal{C}_{s,t}^{\circ}\right) 
	\right)^\circ \right] dt,
\end{equation}
Note that
\begin{eqnarray*}
\rho^s \big(\mathcal{Z}_{0,\mathcal{C}}^{\diamondsuit}(\mathcal{F}), u
\big) &= &\prod_{i=1}^N h^{-s/N} (C_i, \mathbf{X}_i^T u) \\& =&
\int_{\mathbb{R}_{+}^N}\prod_{i=1}^N
\left[u\in \{h^{-s/N}(C_i,{\bf X}_i^T\cdot)>t_i\}\right] dt.
\end{eqnarray*}
For each $i=1,\ldots,N$, we have
\begin{equation*}
\left\{ h^{-s/N} ( C_i,
	\mathbf{X}_i^T \cdot) > t_i \right\}
= \left\{ t_i^{N/s} h (\mathbf{X}_iC_i,
	\cdot) < 1 \right\} 
= \left( t_i^{N/s} \mathbf{X}_i 
	C_i \right)^\circ.
\end{equation*}
By Lemma \ref{lemma:XC_polar},
\begin{eqnarray*}
\bigcap_{i=1}^N \left( t_i^{N/s} \mathbf{X}_i C_i \right)^\circ =
\left({\bf X} \mathcal{B}_1^N(\mathcal{C}_{s,t}^{\circ})\right)^\circ.
\end{eqnarray*}
Therefore,
\begin{align*}
\rho^s \big(\mathcal{Z}_{0,\mathcal{C}}^{\diamondsuit}(\mathcal{F}), u \big) &=
\int_{\mathbb{R}_{+}^N} \left[ u \in \bigcap\nolimits_{i=1}^N \left\{
  h^{-s/N} (C_i, \mathbf{X}_i^T y) > t_i \right\} \right] \, dt \\ &=
\int_{\mathbb{R}_{+}^N} \left[ u \in \left( \mathbf{X} \mathcal{B}_1^N
  \left(\mathcal{C}_{s,t}^{\circ}\right) \right)^\circ \right] \, dt
\end{align*}
Let $\xi$ be a standard Gaussian random vector in $\mathbb{R}^n$ and
$s\in(0,n)$.  Using Lemma \ref{lemma:Gaussian_vol} and \eqref{eqn:rhoZ0C},
\begin{eqnarray}
  \label{eqn:Gauss_level}
 \mathbb{E}_{\xi}
  \rho^{s}(\mathcal{Z}_{0,\mathcal{C}}^{\diamondsuit}(\mathcal{F}),
  \xi) = \int_{\mathbb{R}_+^N}\gamma_n \left(\left(
  \mathbf{X} \mathcal{B}_1^N \left(\mathcal{C}_{s,t}^{\circ}\right)
  \right)^\circ\right)dt.
  \end{eqnarray}
Assume first that
\begin{equation}
\label{eqn:Z0fin}
\mathbb{E}\abs{\mathcal{Z}_{0,\mathcal{C}}^{\diamondsuit}(\mathcal{F})}=\omega_n\mathbb{E}\int_{S^{n-1}}
\rho^n(\mathcal{Z}_{0,\mathcal{C}}^{\diamondsuit}(\mathcal{F}),u)du<\infty.
\end{equation}Then $\rho(\mathcal{Z}_{0,\mathcal{C}}^{\diamondsuit}(\mathcal{F}),\cdot)
\in L_n(S^{n-1},\sigma)$ a.s.. Arguing as in the proof of Lemma
\ref{lemma:Gaussian_vol},
\begin{equation}
  \label{eqn:int_conv}
  \int_{S^{n-1}}\rho^s(\mathcal{Z}_{0,\mathcal{C}}^{\diamondsuit}(\mathcal{F}),u)
  du\rightarrow
  \int_{S^{n-1}}\rho^n(\mathcal{Z}_{0,\mathcal{C}}^{\diamondsuit}(\mathcal{F}),u)
  du \quad \text{(a.s.)},
\end{equation}
and the convergence is dominated by
$1+\omega_n^{-1}\abs{\mathcal{Z}_{0,\mathcal{C}}^{\diamondsuit}(\mathcal{F})}$
(cf. \eqref{eqn:mon_dom}). Thus, using \eqref{eqn:sphere_Gauss}, we get
\begin{eqnarray*}
  \mathbb{E}\abs{\mathcal{Z}_{0,\mathcal{C}}^{\diamondsuit}(\mathcal{F})}
  &=&\omega_n\mathbb{E}_{{\bf X}}\lim_{s\rightarrow
    n^{-}}\int_{S^{n-1}}
  \rho^{s}(\mathcal{Z}_{0,\mathcal{C}}^{\diamondsuit}(\mathcal{F}),u)du\\ &
  = & \omega_n\lim_{s\rightarrow n^{-}}\mathbb{E}_{{\bf
      X}}\int_{S^{n-1}}
  \rho^{s}(\mathcal{Z}_{0,\mathcal{C}}^{\diamondsuit}(\mathcal{F}),u)du\\ &=&
  \omega_n \lim_{s\rightarrow n^{-}}b^{-1}_{n,s}\mathbb{E}_{{\bf
      X}}
  \mathbb{E}_{\xi}\rho^{s}(\mathcal{Z}_{0,\mathcal{C}}^{\diamondsuit}(\mathcal{F}),\xi).
\end{eqnarray*}
Applying \eqref{eqn:Gauss_level} gives the proposition when
$\mathbb{E}\abs{\mathcal{Z}_{0,\mathcal{C}}^{\diamondsuit}(\mathcal{F})}$
is finite. The proposition also remains valid when
$\mathbb{E}\abs{\mathcal{Z}_{0,\mathcal{C}}^{\diamondsuit}(\mathcal{F})}$
is infinite. Indeed, we can replace $S^{n-1}$ in \eqref{eqn:int_conv}
by
$\{\rho(\mathcal{Z}_{0,\mathcal{C}}^{\diamondsuit}(\mathcal{F}),\cdot)\geq
1\}$, in which case the convergence is monotone and both sides of
\eqref{eqn:Z0C_vol} are divergent.
\end{proof}

\subsection{Volume via Gaussian measure for $p<0$}

\label{section:radial_reps}

We will start with a volume formula for the non-random bodies
$Z_{p,C}^{\diamondsuit}(f)$. 

\begin{prop}
  \label{prop:rhoZp}
  Let $f\in \mathcal{P}_n$ and $C\in \mathcal{K}_s^{m}$, where $m\geq
  1$. Let $p\in (-1,0)$ and set $n(p)=n/\abs{p}\in \mathbb{N}$.  Let
  ${\bf X}$ be an $n\times n(p)m$ random matrix with independent
  columns distributed according to $f$.  For $\ell \in \mathbb{N}$,
  let $p_{\ell}=p(1-1/(\ell n))$. For $t_1,\ldots,t_{n(p)}>0$ and
  $\ell\in \mathbb{N}$, let
  $\mathcal{C}_{t,p_{\ell}}^{\circ}=((t_{1}^{1/\abs{p_{\ell}}}
  C)^{\circ},\ldots,(t_{n(p)}^{1/\abs{p_{\ell}}}C)^{\circ})$.  Then
\begin{equation}
  \abs{Z_{p,C}^{\diamondsuit}(f)}=\lim_{\ell\rightarrow \infty}b^{-1}_{n,n-1/\ell}
  \int_{\mathbb{R}_{+}^{n(p)}}
    \mathbb{E}_{{\bf X}}\gamma_n\left(\left({\bf X}
      \mathcal{B}_{1}^{n(p)}(\mathcal{C}_{t,p_{\ell}}^{\circ})\right)^{\circ}\right)dt.
\end{equation}
\end{prop}

\begin{proof}Fix $k\in \mathbb{N}$.  Let ${\bf X}_1,\ldots,{\bf X}_k$ be independent $n\times m$ random
  matrices with independent columns drawn from $f$.  We will first
  show that for $u\in \mathbb{R}^n\backslash\{0\}$,
  \begin{equation}
    \label{eqn:rhoZp}
    \rho^{k\abs{p}}(Z_{p,C}^{\diamondsuit}(f),u)=
    \int_{\mathbb{R}_{+}^k} \mathbb{E}\left[u\in
      \left({\bf X}\mathcal{B}_{1}^{k}(\mathcal{C}_{t,p}^{\circ})
      \right)^{\circ} \right] dt.
  \end{equation}
  Note that
  \begin{eqnarray*}
    \rho^{k\abs{p}}(Z_{p,C}^{\diamondsuit}(f),u) =  \left(\mathbb{E}_{{\bf
        X}_1}h^{-\abs{p}}(C,{\bf X}_1^T u) \right)^{k} = 
    \mathbb{E}_{{\bf X}_1}\cdots \mathbb{E}_{{\bf X}_k}\prod_{i=1}^k
    h^{-\abs{p}}(C,{\bf X}_i^T u)
  \end{eqnarray*}
and
\begin{equation*}
  \prod_{i=1}^k h^{-\abs{p}}(C, {\bf X}_i^Tu)=
  \int_{\mathbb{R}_{+}^{k}}\prod_{i=1}^k
  \left[u\in\{h^{-\abs{p}}(C, {\bf X}_i^T
    \cdot)>t_i\}\right]dt.
\end{equation*}
For each $i=1,\ldots,k$,
\begin{equation*}
\{h^{-\abs{p}}(C, {\bf X}_i^T \cdot)>t_i\}=\{h({\bf X}_iC,
\cdot)<t_i^{-1/\abs{p}}\} = \left(t_i^{1/\abs{p}}{\bf
  X}_iC\right)^{\circ}.
\end{equation*}
By Lemma \ref{lemma:XC_polar},
\begin{eqnarray*}
\bigcap\nolimits_{i=1}^k
\left(t_i^{1/\abs{p}}{\bf X}_iC\right)^{\circ}
=\left({\bf X}\mathcal{B}_{1}^{k}(\mathcal{C}_{t,p}^{\circ})\right)^{\circ}.
\end{eqnarray*}
Therefore,
\begin{equation*}
\mathbb{E}_{{\bf X}_1}\cdots \mathbb{E}_{{\bf X}_k} \prod_{i=1}^k
  h^{-\abs{p}}(C, {\bf X}_i^Tu)=\int_{\mathbb{R}_{+}^k} \mathbb{E}_{\bf X}\left[u\in
    \left({\bf X}\mathcal{B}_{1}^{k}(\mathcal{C}_{t,p}^{\circ})
    \right)^{\circ} \right] dt,
\end{equation*}
which implies \eqref{eqn:rhoZp}. If $\xi$ is a standard Gaussian
vector in $\mathbb{R}^n$, then
\begin{eqnarray}
  \label{eqn:rhoxi}
\mathbb{E}_{\xi} \rho^{k\abs{p}}(Z^{\diamondsuit}_{p,C}(f),\xi)
=\int_{\mathbb{R}_{+}^k}
    \mathbb{E}_{{\bf X}} \gamma_n\left(\left({\bf X}
      \mathcal{B}_{1}^{k}(\mathcal{C}_{t,p}^{\circ})\right)^{\circ}\right)dt.
\end{eqnarray}
Note that $n(p)=\frac{n}{\abs{p}}=\frac{n-1/\ell}{\abs{p_{\ell}}}$.
It remains to apply Lemma \ref{lemma:Gaussian_vol} with
$K=Z_{p,C}^{\diamondsuit}(f)$ and the increasing sequence
$K_{\ell}=Z_{p_{\ell},C}^{\diamondsuit}(f)$ (cf. Lemma
\ref{lemma:Zbasic}(a)). With an eye on \eqref{eqn:rhoxi} with
$p_{\ell}$ and $n(p)$ in place of $p$ and $k$, respectively, we
conclude by
\begin{eqnarray*}
  \abs{Z_{p,C}^{\diamondsuit}(f)} &= &\omega_n\lim_{\ell\rightarrow
    \infty}b_{n,n-1/\ell}^{-1}\mathbb{E}_{\xi}\rho^{n(p)
    \abs{p_{\ell}}}(Z_{p_{\ell},C}^{\diamondsuit}(f),\xi).
\end{eqnarray*}
\end{proof}

\subsection{Radial function representation for $p<0$}

The volume formulas for
$\mathcal{Z}_{0,\mathcal{C}}^{\diamondsuit}(\mathcal{F})$ and
$Z_{p,C}^{\diamondsuit}(f)$ each rely on a representation of the
radial function as a mixture of indicator functions of
origin-symmetric convex bodies. In this subsection, we develop an
analogous representation for the radial function of the empirical
bodies $\mathcal{Z}_{p,\mathcal{C}}^{\diamondsuit}(\mathcal{F})$ for
$p<0$ and $n/\abs{p}\in \mathbb{N}$. A similar volume formula for
$\mathcal{Z}_{p,\mathcal{C}}^{\diamondsuit}(\mathcal{F})$ holds but
the notation becomes lengthy, so we will derive only the radial
function for later use.

To fix the notation, for $k\in \mathbb{N}$, we let ${\bf
  \underline{k}}=(k_1,\ldots,k_N)\in [k]^N$ and define $S({\bf
  \underline{k}})=k_1+\ldots+k_N$ and $m({\bf
  \underline{k}})=\{i\in[N]:k_i\neq 0\}$; we write $\abs{m({\bf
    \underline{k}})}$ for the cardinality of $m({\bf \underline{k}})$.

\begin{prop}
  \label{prop:rhoZpN}
  Let $\mathcal{C}=(C_1,\ldots,C_N)$, $\mathcal{F}$ and ${\bf X}$ be
  as in Lemma \ref{lemma:emp_radial}.  Let $p\in(-1,0)$ and
  $k\in\mathbb{N}$.  Then for $u\in \mathbb{R}^n\backslash\{0\}$,
\begin{eqnarray*}
  \rho^{k\abs{p}}(\mathcal{Z}_{p,\mathcal{C}}^{\diamondsuit}(\mathcal{F}),u) = 
  N^k\sum_{\substack{{{\bf \underline{k}}}\in [k]^N \\S({\bf \underline{k}})=k}}{k \choose {\bf
        \underline{k}}}\int\limits_{\mathbb{R}_{+}^{\abs{m({\bf \underline{k}})}}} \left[u\in
      \left({\bf X}_{{\bf
          \underline{k}}}\mathcal{B}_{1}^{\abs{m({\bf \underline{k}})}}(
  \mathcal{C}_{{\bf \underline{k}},t,p}^{\circ})\right)^{\circ} \right] dt,
\end{eqnarray*} 
where ${k\choose {\bf \underline{k}}}=\frac{k!}{k_1!\cdots k_N!}$, ${\bf X}_{{\bf \underline{k}}}=[{\bf
      X}_{k_i}]_{i\in m({\bf k})}$ and $\mathcal{C}_{{\bf
    \underline{k}},t,p}^{\circ}=((t_{i}^{\frac{1}{k_i\abs{p}}}C_i)^{\circ})_{i\in
  m({\bf \underline{k}})}$.
\end{prop}

\begin{proof}
Using the fact that $k\in \mathbb{N}$, we have for any $u\in
\mathbb{R}^n$,
  \begin{eqnarray*}
    \rho^{k\abs{p}}({\bf X}^{-T}[\mathcal{B}_{p}^N(\mathcal{C})],u)  &
    = & \sum_{\substack{{\bf \underline{k}}\in [k]^N\\ S({\bf \underline{k}})=k}}{k \choose
      {\bf \underline{k}}} \prod_{i\in m({\bf \underline{k}})}
    h^{-k_i\abs{p}}(C_i, {\bf X}_i^T u).
    \end{eqnarray*} 
Fix ${\bf \underline{k}}=(k_1,\ldots,k_N)$ with $S({\bf \underline{k}}) = k$. Then
\begin{equation*}
  \prod_{i\in m({\bf \underline{k}})} h^{-k_i\abs{p}}(C_i, {\bf X}_i^Tu)=
  \int_{\mathbb{R}_{+}^{\abs{m({\bf \underline{k}})}}}\prod_{i\in m({\bf \underline{k}})}
  \left[u\in\{h^{-k_i\abs{p}}(C_i, {\bf X}_i^T
    \cdot)>t_i\}\right]dt.
\end{equation*}
For each $i\in m({\bf \underline{k}})$, 
\begin{equation*}
\{h^{-k_i\abs{p}}(C_i, {\bf X}_i^T \cdot)>t_i\}=\{h({\bf X}_iC_i,
\cdot)<t_i^{-\frac{1}{k_i\abs{p}}}\} = \left(t_i^{\frac{1}{k_i\abs{p}}}{\bf
  X}_iC_i\right)^{\circ}.
\end{equation*}
By Lemma \ref{lemma:XC_polar},
\begin{eqnarray*}
\bigcap\nolimits_{i\in m({\bf \underline{k}})}
\left(t_i^{\frac{1}{k_i\abs{p}}}{\bf X}_iC_i\right)^{\circ}
=\left({\bf X}_{{\bf \underline{k}}}\mathcal{B}_{1}^{\abs{m({\bf
      \underline{k}})}}(\mathcal{C}_{{\bf
      \underline{k}},t,p}^{\circ})\right)^{\circ}.
\end{eqnarray*}
Thus the proposition follows from
\begin{equation*}
  \prod_{i\in m({\bf \underline{k}})} h^{-k_i\abs{p}}(C_i, {\bf
      X}_i^Tu)=\int_{\mathbb{R}_{+}^{\abs{m({\bf \underline{k}})}}} \left[u\in \left({\bf X}_{{\bf
        \underline{k}}}\mathcal{B}_{1}^{\abs{m({\bf
      \underline{k}})}}(\mathcal{C}_{{\bf
      \underline{k}},t,p}^{\circ})\right)^{\circ} \right] dt.
\end{equation*}
\end{proof}

\section{Main proofs}

\label{section:main}

\begin{proof}[Proof of Theorem \ref{thm:ZpC+}]
Suppose that ${\bf X}$ and ${\bf X}^{\#}$ are $n\times M$ random
matrices with independent columns drawn from $\mathcal{F}=
(f_{ij})\subseteq \mathcal{P}_n$ and $\mathcal{F}^{\#}=(f_{ij}^*)$
respectively, where $M=m_1+\ldots+m_N$.  Suppose that each $f_{ij}$ is
supported on a Euclidean ball $RB_2^n$. Denote the expectation in
${\bf X}$ and ${\bf X}^{\#}$ by $\mathbb{E}_{{\bf X}}$ and
$\mathbb{E}_{{\bf X}^{\#}}$, respectively.

We will use Theorem \ref{thm:CEFPP}, combined with the volume formulas
for $\mathcal{Z}^{\diamondsuit}_{p,\mathcal{C}}(\mathcal{F})$ as
indicated. For $p\geq 1$, we have by Remark \ref{remark:Zpcirc},
\begin{eqnarray*}
\mathbb{E}\abs{\mathcal{Z}^{\diamondsuit}_{p,\mathcal{C}}(\mathcal{F})}
&=& \mathbb{E}_{{\bf X}}\abs{N^{1/p}({\bf
    X}\mathcal{B}_q^N(\mathcal{C}^{\circ}))^{\circ}}\\ &\leq &
\mathbb{E}_{{\bf X}^{\#}}\abs{N^{1/p}({\bf
    X}^{\#}\mathcal{B}_q^N(\mathcal{C}^{\circ}))^{\circ}}\\
& = & \mathbb{E}\abs{\mathcal{Z}^{\diamondsuit}_{p,\mathcal{C}}(\mathcal{F}^{\#})}.
\end{eqnarray*}
For $p\in (0,1)$, using Proposition \ref{prop:NT_ext} and Fubini's theorem,
\begin{eqnarray*}
  \mathbb{E}
  \abs{\mathcal{Z}_{p,\mathcal{C}}^{\diamondsuit}(\mathcal{F})} &=&
  a_{N,n,p}\mathbb{E}_{W} \mathbb{E}_{{\bf X}} {\sqrt{w_1\cdots
      w_N}}\abs{\bigl({\bf X}
    \mathcal{B}_2^N(\mathcal{C}_W^{\circ})\bigr)^{\circ}}\\
&\leq &   a_{N,n,p}\mathbb{E}_{W} \mathbb{E}_{{\bf X}^{\#}} {\sqrt{w_1\cdots
      w_N}}\abs{\bigl({\bf X}^{\#}
    \mathcal{B}_2^N(\mathcal{C}_W^{\circ})\bigr)^{\circ}}\\
& = &   \mathbb{E}
  \abs{\mathcal{Z}_{p,\mathcal{C}}^{\diamondsuit}(\mathcal{F}^{\#})}.
\end{eqnarray*}
For $p=0$, we apply Proposition \ref{prop:rhoZ0C} to get
\begin{eqnarray*}
\mathbb{E}
  \abs{\mathcal{Z}_{0,\mathcal{C}}^{\diamondsuit}(\mathcal{F})} &=&
\lim_{s\rightarrow
    n^{-}}\int_{\mathbb{R}_{+}^N}\mathbb{E}_{{\bf X}}\gamma_n\left(({\bf
    X}\mathcal{B}_1^N(\mathcal{C}_{s,t}^{\circ}))^{\circ}\right)dt\\
&\leq & \lim_{s\rightarrow
    n^{-}}\int_{\mathbb{R}_{+}^N}\mathbb{E}_{{\bf X}^{\#}}\gamma_n\left(({\bf
    X}^{\#}\mathcal{B}_1^N(\mathcal{C}_{s,t}^{\circ}))^{\circ}\right)dt\\
&=&\mathbb{E}
  \abs{\mathcal{Z}_{0,\mathcal{C}}^{\diamondsuit}(\mathcal{F}^{\#})}.
\end{eqnarray*}

When the $C_i$'s are identical, we have by Proposition
\ref{prop:ZpCUI},
\begin{equation}
  \abs{Z_{p,C}^{\diamondsuit}(f)}=\lim_{N\rightarrow\infty}\mathbb{E}
  \abs{\mathcal{Z}_{p,\mathcal{C}_N}^{\diamondsuit}({\mathcal{F}_N})}, 
\end{equation}
which proves \eqref{eqn:ZpC+} for $f$ compactly supported. For a general $f\in
\mathcal{P}_n$, we define $\{\phi^{(k)}\}$ as in Lemma \ref{lemma:Zbasic}.
By Fatou's lemma and the compactly supported case,
\begin{eqnarray*}
  \abs{Z_{p,C}^{\diamondsuit}(f)} \leq 
  \liminf_{k\rightarrow \infty}\abs{Z_{p,C}^{\diamondsuit}(\phi^{(k)})}
   \leq   \liminf_{k\rightarrow \infty}\abs{Z_{p,C}^{\diamondsuit}((\phi^{(k)})^*)}
  =  \abs{Z_{p,C}^{\diamondsuit}(f^*)},
\end{eqnarray*}
where the last equality holds as each set
$Z_{p,C}^{\diamondsuit}((\phi^{(k)})^*)$ is a Euclidean ball
and the convergence is ensured by  \eqref{eqn:contraction}.

Lastly, we turn to the case when $\mathcal{F}=(f_{ij})$ consists of
functions that are not supported on a common compact set. In the
notation of Lemma \ref{lemma:Zbasic}(d), we set
$\varphi_{ij}^{(k)}=f_{ij}|_{kB_2^n}$ and
$\phi_{ij}^{(k)}=\varphi_{ij}^{(k)}/\int \varphi_{ij}^{(k)}$ and set
$\mathcal{F}_k=(\phi_{ij}^{(k)})$. Then
\begin{eqnarray*}
  \mathbb{E}\abs{\mathcal{Z}_{p,C}^{\diamondsuit}(\mathcal{F}_k)}
  &=&\int_{((\mathbb{R}^n)^m)^N}\int_{S^{n-1}}\left(\frac{1}{N}\sum_{i=1}^N
  h^p(C,(\langle
  x_{ij},u\rangle)_{j=1}^{m_i})\right)^{n/p}\prod_{i,j}
  \phi^{(k)}_{ij}(x_{ij})du d\overline{x}.
\end{eqnarray*}
Using $\int \varphi_{ij}^{(k)}\rightarrow \int f_{ij}=1$ and monotone
convergence for $\varphi_{ij}^{(k)}$,
\begin{eqnarray*}
\mathbb{E}\abs{\mathcal{Z}_{p,\mathcal{C}}^{\diamondsuit}(\mathcal{F}_k)}
=\lim_{k\rightarrow \infty}\mathbb{E}
\abs{\mathcal{Z}_{p,\mathcal{C}}^{\diamondsuit}(\mathcal{F}_k)}\leq
\lim_{k\rightarrow \infty}\mathbb{E}\abs{\mathcal{Z}_{p,\mathcal{C}}^{\diamondsuit}(\mathcal{F}_k^{\#})}
=\mathbb{E}\abs{\mathcal{Z}_{p,\mathcal{C}}^{\diamondsuit}(\mathcal{F}^{\#})}.
\end{eqnarray*}

\end{proof}

\begin{proof}[Proof of Theorem \ref{thm:Zp+}]
Taking $C=C_i=[-1,1]$ and $\mathcal{F}=(f)$ gives
$\mathcal{Z}^{\diamondsuit}_{p,N}(f)=
\mathcal{Z}^{\diamondsuit}_{p,\mathcal{C}}(\mathcal{F})$ and
$Z_{p,C}^{\diamondsuit}(f)=Z_{p}^{\diamondsuit}(f)$, hence Theorem
\ref{thm:Zp+} follows from Theorem \ref{thm:ZpC+}.
\end{proof}


\begin{proof}[Proof of Theorem \ref{thm:ZpC-}]
Let ${\bf X}$ and ${\bf X}^{\#}$ be $n\times n(p)m$ random matrices
with i.i.d.  columns drawn from $f$ and $f^{*}$, respectively.  By
Proposition \ref{prop:rhoZp} and Theorem
\ref{thm:CEFPP}, \begin{eqnarray*}
  \abs{Z_{p,C}^{\diamondsuit}(f)}&=&\lim_{\ell\rightarrow
    \infty}b^{-1}_{n,n-1/\ell} \int_{\mathbb{R}_{+}^{n(p)}}
  \mathbb{E}_{{\bf X}}\gamma_n\left(\left({\bf X}
  \mathcal{B}_{1}^{n(p)}(\mathcal{C}_{t,p_{\ell}}^{\circ})\right)^{\circ}\right)dt\\ &
  \leq &\lim_{\ell\rightarrow \infty}b^{-1}_{n,n-1/\ell}
  \int_{\mathbb{R}_{+}^{n(p)}} \mathbb{E}_{{\bf
      X}^{\#}}\gamma_n\left(\left({\bf X}^{\#}
  \mathcal{B}_{1}^{n(p)}(\mathcal{C}_{t,p_{\ell}}^{\circ})\right)^{\circ}\right)dt\\ &
  = & \abs{Z_{p,C}^{\diamondsuit}(f^*)}.
\end{eqnarray*}

Next, we prove \eqref{eqn:empZpC-}.  Fix origin-symmetric convex
bodies $C_1,\ldots,C_N$ with $\mathop{\rm dim}(C_i)=m_i\geq n+1$.  Set
$M=m_1+\ldots+m_N$. Suppose that ${\bf X}$ and ${\bf X}^{\#}$ are
$n\times M$ random matrices with independent columns drawn from
$\mathcal{F}=(f_{ij})$ and $\mathcal{F}^{\#}=(f^*_{ij})$,
respectively. Fix $k\in \mathbb{N}$ and $p\in[-1,0)$ with
  $k\abs{p}<n$.

Assume first that $f$ is supported on a Euclidean ball $RB_2^n$. 
By Proposition \ref{prop:ZpCUI},
\begin{equation*}
\mathbb{E}\abs{\mathcal{Z}_{p,\mathcal{C}}^{\diamondsuit}(\mathcal{F})}=\omega_n\mathbb{E}_{{\bf
    X}}\int_{S^{n-1}} \rho^n({\bf
  X}^{-T}[\mathcal{B}_p^N(\mathcal{C})],u)du <\infty.
\end{equation*}
Applying Proposition \ref{prop:rhoZpN} for a standard Gaussian random
vector $\xi$ in $\mathbb{R}^n$, we have
\begin{eqnarray*}
  \mathbb{E}_{\xi}
  \rho^{k\abs{p}}(\mathcal{Z}_{p,\mathcal{C}}^{\diamondsuit}(\mathcal{F}),
  \xi) = \sum_{\substack{{\bf \underline{k}}\in [k]^N \\S({\bf
        \underline{k}})=k}}{k \choose {\bf
      \underline{k}}}\int\limits_{\mathbb{R}_{+}^{\abs{m({\bf
          \underline{k}})}}} \gamma_n\left(\left({\bf X}_{{\bf
      \underline{k}}}\mathcal{B}_{1}^{\abs{m({\bf
        \underline{k}})}}(\mathcal{C}_{{\bf
      \underline{k}},t}^{\circ})\right)^{\circ} \right) dt.
  \end{eqnarray*}
Fix ${\bf \underline{k}}=(k_1,\ldots,k_N)\in [k]^N$ with $S({\bf
  \underline{k}})=k$ and $t_i\in (0,\infty)$ for $i\in m({\bf
  \underline{k}})$. By Theorem \ref{thm:CEFPP},
\begin{eqnarray*}
  \mathbb{E}_{\bf X_{{\bf
      \underline{k}}}} \gamma_n\left(\left({\bf X}_{\bf
    \underline{k}}\mathcal{B}_{1}^{\abs{m({\bf {\bf
          \underline{k}}})}}(\mathcal{C}_{{\bf
      \underline{k}},t}^{\circ})\right)^{\circ} \right) \leq
  \mathbb{E}_{\bf X_{{\bf
      \underline{k}}}^{\#}} \gamma_n\left(\left({\bf X}_{{\bf
      \underline{k}}}^{\#}\mathcal{B}_{1}^{\abs{m({\bf
        \underline{k}})}}(\mathcal{C}_{{\bf
      \underline{k}},t}^{\circ})\right)^{\circ} \right).
\end{eqnarray*}
Consequently,
\begin{equation}
  \label{eqn:hash}
  \mathbb{E}_{{\bf X}} \mathbb{E}_{\xi}
  \rho^{k\abs{p}}(\mathcal{Z}_{p,\mathcal{C}}^{\diamondsuit}(\mathcal{F}),
  \xi) \leq \mathbb{E}_{{\bf X}^{\#}} \mathbb{E}_{\xi}
  \rho^{k\abs{p}}(\mathcal{Z}_{p,\mathcal{C}}^{\diamondsuit}(\mathcal{F}^{\#}),
  \xi).
\end{equation}
As in the proof of Proposition \ref{prop:rhoZp}, when $n/\abs{p}\in
\mathbb{N}$, we choose $p_{\ell}\in \mathbb{Q}\cap(p,0)$ such that
$\frac{n}{\abs{p}}=\frac{n-1/\ell}{\abs{p_{\ell}}}$ for $j\in
\mathbb{N}$. For $u\in S^{n-1}$, we have
\begin{equation*}
\rho(\mathcal{Z}_{p_{\ell},\mathcal{C}}^{\diamondsuit}(\mathcal{F}),u)
\rightarrow
\rho(\mathcal{Z}_{p,\mathcal{C}}^{\diamondsuit}(\mathcal{F}),u)\quad
{\text{(a.s.)}}.
\end{equation*}
As in the proof of Lemma \ref{lemma:Gaussian_vol}, using
\eqref{eqn:mon_dom} with
$K_{\ell}=\mathcal{Z}_{p_{\ell},\mathcal{C}}^{\diamondsuit}(\mathcal{F})$, we have
\begin{equation*}
  \int_{S^{n-1}}\rho^{\ell
    \abs{p_{\ell}}}(\mathcal{Z}_{p_{\ell},\mathcal{C}}^{\diamondsuit}(\mathcal{F}),u)du
  {\rightarrow}
  \int_{S^{n-1}}\rho^{n}(\mathcal{Z}_{p,\mathcal{C}}^{\diamondsuit}(\mathcal{F}),u)du\quad {\text{(a.s.)}}
\end{equation*}
and the convergence is dominated by
$1+\omega_n^{-1}\abs{\mathcal{Z}_{p,\mathcal{C}}^{\diamondsuit}(\mathcal{F})}.$
Thus,
\begin{eqnarray*}
  \mathbb{E}_{{\bf
      X}}\abs{\mathcal{Z}_{p,\mathcal{C}}^{\diamondsuit}(\mathcal{F})}
  &=&\omega_n\mathbb{E}_{{\bf X}}\lim_{\ell\rightarrow
    \infty}\int_{S^{n-1}}
  \rho^{\ell\abs{p_{\ell}}}(\mathcal{Z}_{p_{\ell},\mathcal{C}}^{\diamondsuit}(\mathcal{F}),u)du\\ &
  = & \omega_n\lim_{\ell\rightarrow \infty}\mathbb{E}_{{\bf
      X}}\int_{S^{n-1}}
  \rho^{n-1/\ell}(\mathcal{Z}_{p_{\ell},\mathcal{C}}^{\diamondsuit}(\mathcal{F}),u)du\\ &=&
  \omega_n \lim_{\ell\rightarrow \infty}b^{-1}_{n,n-1/\ell}\mathbb{E}_{{\bf
      X}}
  \mathbb{E}_{\xi}\rho^{n-1/\ell}(\mathcal{Z}_{p_{\ell},\mathcal{C}}^{\diamondsuit}(\mathcal{F}),\xi),
\end{eqnarray*}
where $b_{n,n-1/\ell}$ is the constant in \eqref{eqn:bns}. The same
identities apply for ${\bf X}^{\#}$ and $\mathcal{F}^{\#}$. Thus applying
\eqref{eqn:hash}, we get
\begin{equation*}
\mathbb{E}\abs{\mathcal{Z}_{p,\mathcal{C}}^{\diamondsuit}(\mathcal{F})}
\leq \mathbb{E} \abs{\mathcal{Z}_{p,\mathcal{C}}^{\diamondsuit}(\mathcal{F}^{\#})}.
\end{equation*}
Lastly, we can remove the assumption that the functions are compactly
supported by arguing as in the proof of Theorem \ref{thm:ZpC+}.
\end{proof}

\begin{proof}[Proof of Corollary \ref{cor:int}]
  By Proposition \ref{prop:IZ} and Theorem \ref{thm:ZpC-},
  \begin{eqnarray}
    \label{eqn:INp<}
    \mathbb{E} \abs{\mathcal{I}_{\abs{p},N}^{\alpha}(f)}= 
    \lim_{m\rightarrow \infty}
    \mathbb{E}\abs{\mathcal{Z}^{\diamondsuit}_{p,\mathcal{C}_{m}^\alpha}
      (\mathcal{F}_{m})}
     \leq  \lim_{m\rightarrow \infty}
    \mathbb{E}\abs{\mathcal{Z}^{\diamondsuit}_{p,\mathcal{C}_{m}^\alpha}
      (\mathcal{F}_{m}^{\#})}
    =\mathbb{E} \abs{\mathcal{I}_{\abs{p},N}^{\alpha}(f^{*})}.
  \end{eqnarray}
  Using \eqref{eqn:INp<} with Proposition \ref{prop:Ipa}, we get
  \begin{eqnarray}
    \label{eqn:Ip<}
    \abs{I_{\abs{p}}^{\alpha}(f)} = 
    \lim_{N\rightarrow
  \infty}\mathbb{E}\abs{ \mathcal{I}_{\abs{p},N}^{\alpha}(f)}
    \leq      \lim_{N\rightarrow
  \infty}\mathbb{E}\abs{ \mathcal{I}_{\abs{p},N}^{\alpha}(f^{*})}
     =  \abs{I_{\abs{p}}^{\alpha}(f^*)}.
  \end{eqnarray}
Finally, we apply Proposition \ref{prop:asinh} and \eqref{eqn:Ip<} for
$p=-1$, to obtain
  \begin{eqnarray*}
    \abs{I(f)}=\lim_{\alpha\rightarrow
      0^+}\abs{(2s_{\alpha})^{-1}I_{\alpha}(f)}\leq 
\lim_{\alpha\rightarrow
      0^+}\abs{(2s_{\alpha})^{-1}I_{\alpha}(f^*)}= \abs{I(f^{*})}.
  \end{eqnarray*}
\end{proof}

\begin{proof}[Proof of Theorem \ref{thm:ball}]
We have reduced Theorems \ref{thm:Zp+} to \ref{thm:ZpC-} and
Corollaries \ref{cor:int}, \ref{cor:If} to a suitable application of
\eqref{eqn:CEFPPa} in Theorem \ref{thm:CEFPP}.  When the convex bodies
$C_1,\ldots,C_N$ are unconditional, we can instead apply  \eqref{eqn:CEFPPb} in
Theorem \ref{thm:CEFPP}.
\end{proof}
\vspace{-.1cm}
\noindent{\bf Acknowledgements:} It is our pleasure to thank Alex
Koldobsky, Monika Ludwig, Franz Schuster and Vlad Yaskin for helpful
discussions. The second and third-named authors learned about the
question of linking the Busemann intersection inequality to the (dual)
Busemann-Petty centroid inequality from Monika Ludwig at the workshop
{\it Invariants in convex geometry and Banach space theory}, held at
the American Institute of Mathematics in August 2012. The third-named
author thanks Gabriel Lip and Jill Ryan for their warm hospitality in
Calgary.
\vspace{-.1cm}

\addcontentsline{toc}{section}{References} \bibliographystyle{plain}
\bibliography{mixture}

\small
\vspace{1cm}

Institute of Mathematics, University of Warsaw, ul. Banacha 2, 02-097
Warszawa, Poland \\Email address: r.adamczak@mimuw.edu.pl\\

Department of Mathematics, Texas A\&M University, College Station,
Texas, 77840 \\Email address: grigoris@math.tamu.edu\\

Mathematics Department, University of Missouri, Columbia, Missouri
65211 \\Email address: pivovarovp@missouri.edu\\

Mathematics Department, University of Missouri, Columbia, Missouri
65211 \\Email address: simanjuntakp@missouri.edu\\

\end{document}